\documentclass[10pt]{article}
\usepackage{amsmath}
\usepackage{amssymb}
\usepackage{amsthm,mathtools}
\usepackage{amscd}
\usepackage{amsfonts}
\usepackage{stfloats}
\usepackage{float}
\usepackage{graphicx}
\usepackage{color}
\usepackage{amsfonts}
\usepackage{amsthm}
\usepackage{psfrag} 
\usepackage{graphicx,url,newverbs}
\usepackage{caption}
\usepackage[ruled,vlined]{algorithm2e}
\usepackage{algorithmic}

\usepackage{amsfonts,verbatim}
\usepackage{amssymb}
\usepackage{amsmath}
\usepackage{subfigure}
\usepackage{psfrag,curves,algorithm2e}
\usepackage{enumitem}

\newcommand{\fig}[1]{\mbox{Figure{~#1}}}

\newcommand{\bx}{{\mathbf{x}}}
\newcommand{\bmu}{{\boldsymbol{\mu}}}
\newcommand{\blambda}{{\boldsymbol{\lambda}}}
\newcommand{\bu}{\mathbf{u}}
\newcommand{\bphi}{\boldsymbol{\phi}}
\newcommand{\bpoly}{\mathbf{p}}
\newcommand{\half}{\frac{1}{2}}

\usepackage{amsthm}
\newtheorem{thm}{Theorem}
\newtheorem{rem}{Remark}
\newtheorem{cor}{Corollary}
\newtheorem{lem}{Lemma}[section]

\usepackage{mathtools}
\newcommand{\verteq}{\rotatebox{90}{$\,=$}}

\usepackage[top=2.8cm,bottom=2.8cm,left=3.5cm,right=3.5cm]{geometry}

\newcommand{\ba}{\begin{array}}
\newcommand{\ea}{\end{array}}
\newcommand{\be}{\begin{equation}}
\newcommand{\ee}{\end{equation}}
\newcommand{\bd}{\begin{displaymath}}
\newcommand{\ed}{\end{displaymath}}
\newcommand{\pa}{\partial}
\newcommand{\f}{\frac}

\usepackage[dvipsnames]{xcolor}

\usepackage{todonotes}

\begin{document}

\title{Universal AMG Accelerated Embedded Boundary Method Without Small Cell Stiffness}

\author{
Zhichao Peng 
\thanks{Department of Mathematics, Michigan State University, East Lansing, MI 48824 U.S.A. Email: {\tt pengzhic@msu.edu}.}
\and
Daniel Appel\"{o}
\thanks{Department of Computational Mathematics, Science and Engineering and Department of Mathematics, Michigan State University, East Lansing, MI 48824 U.S.A. Email: {\tt appeloda@msu.edu}.} 
\and
Shuang Liu 
\thanks{Department of Mathematics, University of California, San Diego, CA 92093, U.S.A.  Email: {\tt shl083@ucsd.edu}.}
}

\date{\today}
\maketitle

\begin{abstract}
	
	\noindent
	We  develop a universally applicable  embedded boundary finite difference method, which results in a symmetric positive definite linear system and does not suffer from small cell stiffness. Our discretization is efficient for the wave, heat and Poisson's equation with Dirichlet boundary conditions. When the system needs to be inverted we can use the conjugate gradient method, accelerated by algebraic multigrid techniques. A series of numerical tests for the wave, heat and Poisson's equation and applications to shape optimization problems verify the accuracy, stability, and efficiency of our method. Our fast computational techniques can be extended to moving boundary problems (e.g. Stefan problem), to the Navier-Stokes equations, and to the Grad-Shafranov equations for which problems are posed on domains with complex geometry and fast simulations are very important. 
	
	\bigskip


\end{abstract}


\section{Introduction}
Fast and accurate simulation of problems arising in engineering and the sciences is of great importance.   Often such problems are posed on domains with complex geometry and the numerical methods used in the simulations must account for this. There are numerous methods that are capable of handling geometry, among them are the finite element method \cite{CourantFEM} and other methods using unstructured grids, overset grid methods  \cite{volkov1968method,starius1977composite,CHESSHIRE:1990ca} and embedded boundary methods. 

In element based methods a volumetric grid must be generated and in both two and three dimensions this can be a time consuming task, especially if the grids are to be of high quality. In methods that use the overset grid (also known as composite grid, or overlapping grid) framework the grids are overset and a Cartesian background grid is coupled through interpolation to local boundary fitted narrow grids near the geometry. The locality of the boundary fitted grids makes the grid generation easier and the quality of the grids are is typically very high. The hole-cutting, the process where the interpolation operators between grids are constructed, can be done efficiently \cite{CHESSHIRE:1990ca} but there are relatively few software packages available. 

In embedded boundary (EB) methods, which is the topic of this paper,  the geometry is represented by curves in two dimensions and surfaces in three dimensions. These curves or surfaces are, as the name suggests, embedded in a uniform Cartesian grid that covers the computational domain. In an embedded boundary method there is no need to generate a grid and the geometry is instead incorporated through modified stencils near the boundary that explicitly incorporate boundary conditions. 

The purpose of this work is to introduce an universal embedded boundary method that can be used to efficiently solve the wave, heat or Poisson's equation with Dirichlet boundary conditions. This is achieved by designing a method that:   
\setlist{nolistsep}
\begin{itemize}[noitemsep]
\item is symmetric and positive definite, so that the conjugate gradient (CG) method can be used,
\item is diagonally dominant and with eigenvalues and eigenvectors that closely resemble those of the periodic problem, so that the conjugate gradient method can be accelerated by algebraic multigrid (AMG) techniques,
\item and, does not suffer from small cell stiffness so that the wave equation can be marched in time by an explicit method.     
\end{itemize}

The basic ingredient to obtaining a symmetric embedded boundary discretization is to use an approximation for the boundary condition that only modifies the diagonal element in the matrix approximating the second derivative. The most straightforward approach is to approximate the second derivative dimension-by-dimension and use linear extrapolation based on the boundary condition and the numerical solution at the interior point to assign the value of the numerical solution at an {\bf outside ghost-point}. In one dimension this modifies the diagonal element from $-2$ to $-(1/\theta+1)$ where $\theta \in (0,1]$ depends on how the boundary cuts the grid. This does not change the symmetry of the matrix and it also does not change its definiteness. It does however change the spectrum of the matrix, introducing an eigenvalue that scale as $1/\theta$. If used directly for the wave equation this approach will thus suffer from small cell stiffness, forcing the stable time-step to be excessively small. This can be mitigated by the local time-stepping proposed by Kreiss, Petersson and Ystr\"{o}m in \cite{kreiss:1940} resulting in a provably stable and efficient embedded boundary method. The same approach was also introduced by Gibou et al. \cite{gibou2002second} for Poisson's equation and the heat equation with implicit time-stepping. Then, as the time-stepping is already implicit the small cell does not reduce the time-step, however the very large eigenvalues that result from the small cells can result in matrices with large condition numbers and care has to be taken when choosing an iterative solver. 

Here we expand on the ideas in \cite{gibou2002second} and  \cite{kreiss:1940} but just as in our earlier work \cite{AppPetEB09} we enforce the Dirichlet boundary conditions by {\bf interpolating to interior boundary points} rather than extrapolate to exterior ghost-points. This subtle yet crucial difference improves
previous second-order accurate approaches by removing the small-cell stiffness problem. Moreover, placing boundary points inside the computational domain allows the solution to be ``single-valued'' for slender geometries, leading to
significant algorithmic simplifications. For geometry that is convex it is still possible to use a line-by-line linear extrapolation to enforce boundary conditions but when the geometry is concave this procedure can break down. For such cases we introduce a combined polynomial and radial basis function interpolation that prevents breakdown. We also propose a simple criterion that can be used to test if the system matrix is SPD.    

Embedded boundary methods have been used to successfully solve a variety of problems from elasticity \cite{WellerShortly1939} to incompressible \cite{Peskin:1972qz} and compressible flows \cite{Pember:1995bk}. Here we do not aim to provide a complete literature survey but rather to mention contributions relevant to the method we introduce. As mentioned above \cite{gibou2002second} develops a symmetric second order method by imposing the boundary condition through linear extrapolation. This method is analyzed and expanded to quadratic boundary treatment in  \cite{jomaa2005embedded}. To obtain second order accuracy in both the solution and its gradient, \cite{ng2009guidelines} proposes a non-symmetric discretization based on a quadratic extrapolation. Adaptive mesh refinement (AMR) techniques for the method in \cite{ng2009guidelines} were considered in \cite{chen2007supra,chen2009numerical}. Finite volume solvers with embedded boundaries using bilinear interpolation were considered in \cite{johansen1998cartesian,schwartz2006cartesian}. Higher order accurate methods for Poisson's and the heat equation can be found in \cite{coco2013finite, gibou2005fourth}. For wave equations the early works by Kreiss, Petersson and Ystr\"{o}m \cite{Kreiss_Petersson_2006,kreiss:1940,kreiss:1292} provided analysis of second order accurate methods with external ghost-points and Dirichlet, Neumann and interface conditions. Higher order accurate methods for the wave equation include \cite{AppPetEB09,FCAD1,Li2004295,Lombard:2004tx,LomPirGelVir07,LyonThesis,FCAD2,Visher:2004iz,Wandzura2004763}.

The rest of the paper is organized as follows. In Section \ref{sec:description}, we first present the overall algorithm. Then, we show the details of how interior boundary points are located, the formulation of the line-by-line and RBF-based interpolation as well as the SPD checking criteria. In Section \ref{sec:num}, the performance of the proposed method is demonstrated through a series of numerical experiments. In Section \ref{sec:conclusion}, we summarize and  conclude.

\section{Universal embedded boundary discretization of the Laplacian}\label{sec:description}
Our goal is to design an embedded boundary finite difference method, which results in a symmetric positive definite (SPD) linear system and does not suffer from small cell stiffness. Our discretization can thus be efficient for the wave, heat and Poisson's equation. We now describe the different components of our method one at a time but note that the entire method is summarized in Algorithm \ref{alg:cut_cell}.

To demonstrate the discretization of the Laplacian operator, consider Poisson's equation in an irregular two-dimensional domain $(x,y)\in\Omega$:
\be \label{eq:peq} 
\nabla\cdot(\beta(x,y)\nabla u) = f(x,y), \quad (x,y)\in \Omega,
\ee 
closed by Dirichlet boundary conditions
\be
\label{eq:peqbcD} u(x,y) = u^{(l)}_{\mathcal{D}}(x,y),\quad (x,y)\in \Gamma_l, \ \  l=1,\ldots, n_{\rm tot}. 
\ee 
The boundary of the domain $\Omega$ is a collection of $n_{\text{tot}}$ smooth
curves $\Gamma_l$. A possible set-up for an interior problem is shown in \fig{\ref{fig:probsetup}}, where one curve encloses the other $n_{\rm tot}-1$  curves.    
\begin{figure}[]
	\begin{center}
	\setlength{\unitlength}{0.8cm} 
	\begin{picture}(7,7) 
		\qbezier(1,1)(3.5,0.5)(2,2)
		\qbezier(2,2)(1.5,2.5)(1,2)
		\qbezier(1,2)(0.2,1.2)(1,1)
		\put(0.6,2.3){$\Gamma_1$} 
		
		\qbezier(5,1)(5.9,1.3)(5,3)
		\qbezier(5,3)(4,5)(4,2.2)
		\qbezier(4,2.2)(4,1)(5,1)
		\put(5.2,3.1){$\Gamma_2$}

		\qbezier(2,4)(1.3,5)(2,6)
		\qbezier(2,6)(3.0,7.0)(3,5.5)
		\qbezier(3,5.5)(3.1,1.8)(2,4)
		\put(3.1,5.3){$\Gamma_3$} 
		
		\qbezier(0.9,0.4)(4.5,0)(6.3,0.4)
		\qbezier(6.3,0.4)(6.6,0.5)(6.6,0.7)
		\qbezier(6.6,0.7)(7,4.5)(6.6,6.3)
		\qbezier(6.6,6.3)(6.5,6.45)(6.3,6.5)
		\qbezier(6.3,6.5)(4.5,7)(0.8,6.5)
		\qbezier(0.8,6.5)(0.63,6.48)(0.5,6.2)
		\qbezier(0.5,6.2)(0,4.5)(0.5,0.9)
		\qbezier(0.5,0.9)(0.55,0.45)(0.9,0.4)
		
		\put(6,6){$\Gamma_4$} 
	\end{picture}	
	\end{center}
	\caption{An illustration of an interior problem with the boundary consisting of four disconnected curves. \label{fig:probsetup}}
\end{figure}
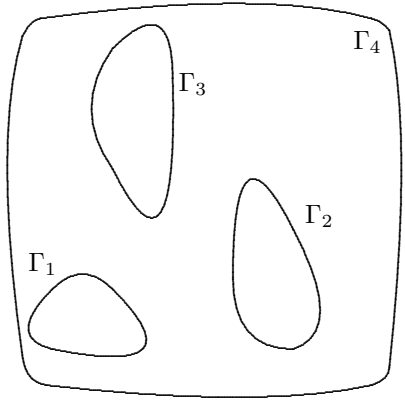

Without loss of generality, we assume that all the boundary interfaces describing the geometry $\Omega$ are contained inside the uniform Cartesian grid (see
\fig{\ref{fig:probsetupdiscrete}})
\bd
(x_i,y_j) = (x_L+(i-1)h,y_L+(j-1)h),\ \ i=1,\ldots,N_x,\, j=1,\ldots,N_y,
\ed
discretizing the rectangular domain $[x_L,x_R] \times [y_L,y_R]$ where $x_L$ and $y_L$ are given and $x_{N_x}=x_R$ and $y_{N_y} = y_R$ are determined so that they align with the grid. 

To this end we will approximate the Laplacian operator by the central difference scheme:
\begin{align}
\nabla\cdot(\beta(x_i,y_j)\nabla u(x_i,y_j))&\approx 
\frac{\beta_{i+\half,j}u_{i+1,j}-(\beta_{i+\half,j}+\beta_{i-\half,j})u_{i,j}+\beta_{i-\half,j}u_{i-1,j}}{h^2}\notag\\
&+\frac{\beta_{i,j+\half}u_{i,j+1}-(\beta_{i,j+\half}+\beta_{i,j-\half})u_{i,j}+\beta_{i,j-\half}u_{i,j-1}}{h^2}.
\label{eq:central_difference}
\end{align}

As mentioned in the introduction, a novelty of our method is to enforce boundary conditions through interpolation to interior boundary points. This avoids small cell stiffness. To identify interior boundary points, we first set up a mask grid function  $m_{i,j}$  defined to be one inside the geometry and zero outside. That is:
\[
m_{i,j} = \begin{cases}
1,&(x_i,y_j)\in\Omega,\\
0,&\mbox{otherwise}.
\end{cases}
\]
For example, if the boundary of $\Omega$ is determined by the signed level-set function $\psi(x,y)$=0, then the value of the mask $m_{i,j}$ follows by the sign of $\psi(x_i,y_j)$. An example of a mask grid function
is shown in \fig{\ref{fig:probsetupdiscrete}}.

We denote grid points inside $\Omega$ but adjacent to the boundary as \emph{boundary  points} and we denote the remaining interior grid points as \emph{computational points}. Precisely we define: 
\begin{itemize}
\item A {\bf boundary 
point} $(x_i,y_j)$ satisfies:
\begin{itemize} 
\item[(1)] $m_{i,j} = 1$,
\item[(2)] $m_{i+1,j} + m_{i-1,j} +m_{i,j+1} +m_{i,j-1} < 4$.
\end{itemize} In other words, $(x_i,y_j)$ is inside $\Omega$, but at least one
of its nearest neighbors is outside. 
\item A {\bf computational point} $(x_i,y_j)$ satisfies: 
\begin{itemize}
\item[(1)] $m_{i,j} = 1$, 
\item[(2)] $m_{i+1,j} + m_{i-1,j} +m_{i,j+1} +m_{i,j-1} = 4$. 
\end{itemize}
In other words, $(x_i,y_j)$ and all its nearest neighbors are all inside $\Omega$.
\end{itemize}

\begin{figure}[htb]
\begin{center}
\subfigure[]{
\setlength{\unitlength}{0.7cm} 
\begin{picture}(7,7) 
\multiput(0,0)(0,0.7){11}{\line(1,0){7}}
\multiput(0,0)(0.7,0){11}{\line(0,1){7}}

\qbezier(1,1)(3.5,0.5)(2,2)
\qbezier(2,2)(1.5,2.5)(1,2)
\qbezier(1,2)(0.2,1.2)(1,1)
\put(0.6,2.3){$\Gamma_1$} 

\qbezier(5,1)(5.9,1.3)(5,3)
\qbezier(5,3)(4,5)(4,2.2)
\qbezier(4,2.2)(4,1)(5,1)
\put(5.2,3.1){$\Gamma_2$}

\qbezier(2,4)(1.3,5)(2,6)
\qbezier(2,6)(3.0,7.0)(3,5.5)
\qbezier(3,5.5)(3.1,1.8)(2,4)
\put(3.1,5.3){$\Gamma_3$} 

\qbezier(0.9,0.4)(4.5,0)(6.3,0.4)
\qbezier(6.3,0.4)(6.6,0.5)(6.6,0.7)
\qbezier(6.6,0.7)(7,4.5)(6.6,6.3)
\qbezier(6.6,6.3)(6.5,6.45)(6.3,6.5)
\qbezier(6.3,6.5)(4.5,7)(0.8,6.5)
\qbezier(0.8,6.5)(0.63,6.48)(0.5,6.2)
\qbezier(0.5,6.2)(0,4.5)(0.5,0.9)
\qbezier(0.5,0.9)(0.55,0.45)(0.9,0.4)

\put(6,6){$\Gamma_4$} 
\end{picture}
}
\subfigure[]{
\setlength{\unitlength}{0.7cm} 
\begin{picture}(7,7) 
\multiput(0,0)(0,0.7){11}{\line(1,0){7}}
\multiput(0,0)(0.7,0){11}{\line(0,1){7}}

\multiput(0.7,0.7)(0.7,0){9}{\circle*{0.15}}
\multiput(2.8,1.4)(0.7,0){2}{\circle*{0.15}}
\multiput(5.6,1.4)(0.7,0){2}{\circle*{0.15}}
\multiput(0.7,2.1)(0.7,0){1}{\circle*{0.15}}
\multiput(2.1,2.1)(0.7,0){3}{\circle*{0.15}}
\multiput(5.6,2.1)(0.7,0){2}{\circle*{0.15}}
\multiput(0.7,2.8)(0.7,0){5}{\circle*{0.15}}
\multiput(5.6,2.8)(0.7,0){2}{\circle*{0.15}}
\multiput(0.7,3.5)(0.7,0){3}{\circle*{0.15}}
\multiput(3.5,3.5)(0.7,0){1}{\circle*{0.15}}
\multiput(4.9,3.5)(0.7,0){3}{\circle*{0.15}}
\multiput(0.7,4.2)(0.7,0){2}{\circle*{0.15}}
\multiput(3.5,4.2)(0.7,0){5}{\circle*{0.15}}
\multiput(0.7,4.9)(0.7,0){2}{\circle*{0.15}}
\multiput(3.5,4.9)(0.7,0){5}{\circle*{0.15}}
\multiput(0.7,5.6)(0.7,0){2}{\circle*{0.15}}
\multiput(3.5,5.6)(0.7,0){5}{\circle*{0.15}}
\multiput(0.7,6.3)(0.7,0){3}{\circle*{0.15}}
\multiput(3.5,6.3)(0.7,0){5}{\circle*{0.15}}

\qbezier(1,1)(3.5,0.5)(2,2)
\qbezier(2,2)(1.5,2.5)(1,2)
\qbezier(1,2)(0.2,1.2)(1,1)
\put(0.6,2.3){$\Gamma_1$} 

\qbezier(5,1)(5.9,1.3)(5,3)
\qbezier(5,3)(4,5)(4,2.2)
\qbezier(4,2.2)(4,1)(5,1)
\put(5.2,3.1){$\Gamma_2$} 

\qbezier(2,4)(1.3,5)(2,6)
\qbezier(2,6)(3.0,7.0)(3,5.5)
\qbezier(3,5.5)(3.1,1.8)(2,4)
\put(3.1,5.3){$\Gamma_3$} 

\qbezier(0.9,0.4)(4.5,0)(6.3,0.4)
\qbezier(6.3,0.4)(6.6,0.5)(6.6,0.7)
\qbezier(6.6,0.7)(7,4.5)(6.6,6.3)
\qbezier(6.6,6.3)(6.5,6.45)(6.3,6.5)
\qbezier(6.3,6.5)(4.5,7)(0.8,6.5)
\qbezier(0.8,6.5)(0.63,6.48)(0.5,6.2)
\qbezier(0.5,6.2)(0,4.5)(0.5,0.9)
\qbezier(0.5,0.9)(0.55,0.45)(0.9,0.4)

\put(6,6){$\Gamma_4$} 
\end{picture}
}
\caption{Embedded boundaries of an interior problem in a rectangular mesh.  (a) Discretization of the geometry without the mask shown. (b) Discretization of the geometry with the mask shown, grid points with a filled circle have $m_{ij}=1$, grid points without a filled circle have $m_{ij}=0$. \label{fig:probsetupdiscrete}}
\end{center}
\end{figure}
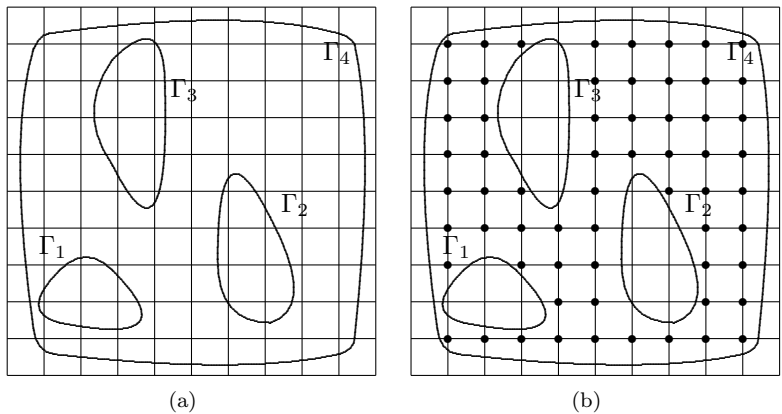

\subsection{Imposing the boundary condition at boundary points}\label{sec:bc_treatment}
We now describe how we use interpolation together with the boundary conditions to assign values to the solution at interior boundary points. We use two strategies, line-by-line interpolation and radial basis interpolation. We first describe the line-by-line approach. 

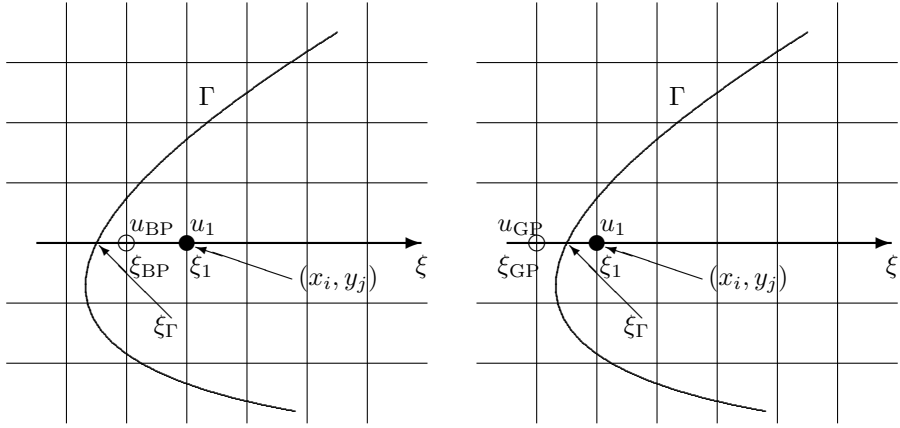
\begin{figure}[htb]
\begin{center}
\setlength{\unitlength}{0.8cm} 
\begin{picture}(7,7) 
\multiput(0,1)(0,1){2}{\line(1,0){7}}
\multiput(0,4)(0,1){3}{\line(1,0){7}}
\multiput(1,0)(1,0){6}{\line(0,1){7}}
\qbezier(5.5,6.5)(-2.5,1.5)(4.8,0.2)
\put(2,3){\circle{0.25}}
\put(3,3){\circle*{0.25}}

\put(2.05,3.2){$u_{\textrm{BP}}$} 
\put(3.05,3.2){$u_{1}$} 

\put(3.2,5.3){$\Gamma$}
\put(2.05,2.5){$\xi_{\textrm{BP}}$} 
\put(3.05,2.5){$\xi_{1}$} 

\put(2.45,1.5){$\xi_{\Gamma}$} 
\put(2.75,1.75){\vector(-1,1){1.2}}

\put(4.75,2.4){\vector(-3,1){1.6}}
\put(4.85,2.3){$(x_i,y_j)$} 

\thicklines
\put(0.5,3){\vector(1,0){6.4}}
\put(6.8,2.5){$\xi$}
\end{picture}
\hspace{0.3cm}
\setlength{\unitlength}{0.8cm} 
\begin{picture}(7,7) 
\multiput(0,1)(0,1){2}{\line(1,0){7}}
\multiput(0,4)(0,1){3}{\line(1,0){7}}
\multiput(1,0)(1,0){6}{\line(0,1){7}}
\qbezier(5.5,6.5)(-2.5,1.5)(4.8,0.2)
\put(1,3){\circle{0.25}}
\put(2,3){\circle*{0.25}}

\put(0.35,3.2){$u_{\textrm{GP}}$} 
\put(2.05,3.2){$u_{1}$} 

\put(0.35,2.5){$\xi_{\textrm{GP}}$} 
\put(2.05,2.5){$\xi_{1}$} 

\put(2.45,1.5){$\xi_{\Gamma}$} 
\put(2.75,1.75){\vector(-1,1){1.2}}

\put(3.2,5.3){$\Gamma$}

\put(3.75,2.4){\vector(-3,1){1.6}}
\put(3.85,2.3){$(x_i,y_j)$}

\thicklines
\put(0.5,3){\vector(1,0){6.4}}
\put(6.8,2.5){$\xi$}
\end{picture}
\caption{Enforcing Dirichlet boundary conditions by a line by line approach using interior boundary points (left) or exterior  ghost points (right).\label{fig:LinebylineD}}
\end{center}
\end{figure}

\subsubsection{Line-by-line interpolation}\label{sec:linebyline}
We describe the line-by-line approach for a case such as the one depicted in the left image of \fig{\ref{fig:LinebylineD}}.

Let $(x_i,y_j)$ be a computational point and $(x_{\textrm{BP}}, y_{\textrm{BP}})=(x_{i-1},y_j)$ be a boundary point associated with the boundary $\Gamma$.  If the point
$(x_{i-2},y_j)$ is outside $\Omega$, 
we introduce a local one-dimensional coordinate system
$\xi$ along the grid line in $x$ passing through $(x_{\textrm{BP}},y_{\textrm{BP}})$. We denote the intersection of the horizontal line $ y = y_j$  and  the boundary $\Gamma$ by $\xi_\Gamma$, and the boundary value at $\xi_\Gamma$ by $u_{\Gamma}$.  The $\xi_\Gamma$ satisfies the scalar equation  $\psi(\xi_\Gamma,y_j) = 0$ and can be found by a root-finding algorithm such as the secant method.

Let $u_{\Gamma}$ be the value of the boundary condition at $(\xi_\Gamma,y_j)$. We introduce an interpolating polynomial
\be \label{eq:Interpolant}
\mathcal{I}_{P}u(\xi) = u_{\Gamma}g_{\Gamma}(\xi) + u_{1} g_{1}(\xi),  
\ee 
where $g_\Gamma$ and $g_1$ are the Lagrange polynomials
\be
g_\Gamma(\xi) = \frac{\xi-\xi_1}{\xi_\Gamma-\xi_1},\; g_1(\xi)=\frac{\xi-\xi_\Gamma}{\xi_1-\xi_\Gamma}.
\ee
Then, the value of the solution at the boundary point $u_{\textrm{BP}}$ can be approximated to second order accuracy by evaluating the interpolant 
\be
\label{eq:DirichletLineByLine} u_{\textrm{BP}} = \mathcal{I}_{P}u(\xi_{\textrm{BP}}) = u_\Gamma g_\Gamma(\xi_{\textrm{BP}})+u_1 g_1(\xi_{\textrm{BP}})= u_\Gamma \frac{\xi_{\textrm{BP}}-\xi_1}{\xi_\Gamma-\xi_1}-u_1\frac{\xi_{\textrm{BP}}-\xi_\Gamma}{\xi_\Gamma-\xi_1}.
\ee

The placement of the boundary point inside the boundary is the subtle yet important distinction from previous
methods like those in \cite{gibou2002second,Kreiss_Petersson_2006,kreiss:1940,kreiss:1292}. 
In previous work, the point is placed outside (and is usually referred to as a ghost point), see the right image of \fig{\ref{fig:LinebylineD}}. Then the linear interpolant will contain a factor  
\bd
\frac{\xi_{\textrm{GP}}-\xi_1}{\xi_\Gamma-\xi_1}
\ed
which can be arbitrarily large when $\xi_\Gamma$ is close to $\xi_{1}$. This causes small-cell stiffness or numerical overflow in the assembly process of the system of equations. 

For Poissons equation the boundary interface can be moved to the interior points if $\frac{\xi_{\textrm{GP}}-\xi_1}{\xi_\Gamma-\xi_1}<\textrm{threshold}\approx O(h)$, \cite{gibou2002second}, however, this will introduce an eigenvalue that scale as $O(\frac{1}{h})$. Such a large eigenvalue will lead to a very  restrictive time step for the wave equation.

In contrast, for the approach suggested above, $\xi_\Gamma \le \xi_{\textrm{BP}} < \xi_1 = \xi_{\textrm{BP}}+h$, and thus  
\bd
|g_{\Gamma}(\xi_{\textrm{BP}})|=\left|\frac{\xi_{\textrm{BP}}-\xi_1}{\xi_\Gamma-\xi_1}\right|\leq 1\;\text{and}\; |g_1(\xi_{\textrm{BP}})|=\left|\frac{\xi_{\textrm{BP}}-\xi_\Gamma}{\xi_\Gamma-\xi_1}\right|\leq 1.
\ed  
Substituting the value of $u_{\textrm{BP}}$ into the central difference approximation for $(\beta u_x)_x$ 
\bd
\frac{\beta_{i-\half,j}u_{\textrm{BP}}-(\beta_{i-\half,j}+\beta_{i+\half,j})u_{ij}+\beta_{i+\half,j}u_{i+1,j}}{h^2},
\ed
we have
\bd
\frac{1}{h^2}\Big(\; (-1+g_1(\xi_{\textrm{BP}}) )\beta_{i-\half,j} u_{ij} -\beta_{i+\half,j}(u_{ij}-u_{i+1,j})+g_\Gamma(\xi_{\textrm{BP}})u_\Gamma\beta_{i-\half,j}\Big).
\ed
Because only the diagonal element is modified and $|g_1(\xi_{\textrm{BP}})|\leq 1$, the resulting linear system is still symmetric and diagonally dominant with correct sign. As a result, the SPD structure of the discrete Laplacian operator is preserved.


\subsection{Radial Basis Function (RBF) interpolation}\label{sec:rbf}
Unfortunately there are some cases when the line-by-line approach cannot be used. For example, when the geometry is non-convex (there is an inward pointing smooth corner), as \fig{\ref{fig:RBF}}, it can happen that the intersection between the grid line and the  boundary does not exist, or it is far away. In \fig{\ref{fig:RBF}}, the stencil \eqref{eq:central_difference} requires that the leftmost interior boundary point is determined by the interpolant in the $x$-direction but the intersection with the boundary along the grid-line may be very far away and would result in a very inaccurate approximation. 
Of course, the value at the boundary point can be specified by interpolating along the $y$-direction. However, as this would result in a non-diagonal modification of the system matrix and break its symmetry, we instead propose an alternative approach. 

 
\begin{figure}[htb]
\begin{center}
\setlength{\unitlength}{0.8cm} 
\begin{picture}(7,7) 
\multiput(0,2)(0,2){1}{\line(1,0){7}}
\multiput(0,4)(0,1){3}{\line(1,0){7}}
\multiput(1,1)(1,0){6}{\line(0,1){6}}
\qbezier(5.5,7)(1,0.75)(0,6)
\put(2,3){\circle{0.25}}
\put(3,3){\circle*{0.25}}
\put(2.65,3.95){\circle*{0.25}}
\put(1.75,3.62){\circle*{0.25}}

\put(1.25,3.25){$u_{\Gamma_1}$} 
\put(2.25,4.3){$u_{\Gamma_2}$} 
\put(2.05,3.2){$u_{\textrm{BP}}$} 
\put(3.05,3.2){$u_{ij}$}

\put(3.8,5.3){$\Gamma$}
\put(2.05,2.5){$\bx_{\textrm{BP}}$} 
\put(3.05,2.5){$\bx_{ij}$} 

\thicklines
\put(0.5,3){\vector(1,0){6.4}}
\put(6.8,2.5){$\xi$}
\end{picture}
\caption{An illustration of the points used to construct the RBF interpolant for evaluating $u_{\textrm{BP}}$.\label{fig:RBF}}
\end{center}
\end{figure}
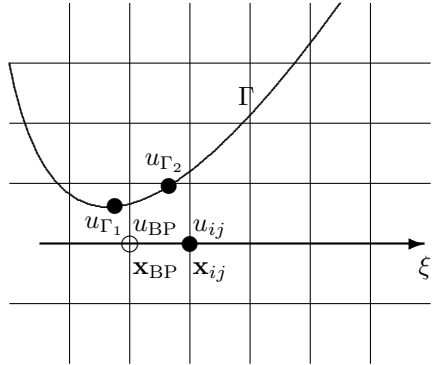

For geometries where the interior boundary point cannot be accurately determined by line-by-line interpolation, we instead use the radial basis function (RBF) interpolation. To do this we  find two suitable distinct points on the boundary of $\Omega$, and  utilize these two points and the interior computational point to interpolate at the interior boundary point with a  polynomial augmented RBF interpolant \cite{wright2003radial}. In this section, we first present the associated RBF-based interpolation and then discuss how to choose the two distinct points on the boundary.

Without loss of generality, we consider the case presented in \fig{\ref{fig:RBF}}. Let $\bx_{ij}=(x_i,y_j)$ be the computational point where we want the approximation to the Laplacian operator. Let $\bx_{\textrm{BP}}=(x_{\textrm{BP}},y_{\textrm{BP}})$ be the  boundary point needed in the stencil in the $x$-direction, and let $\bx_{\Gamma_1}$
and $\bx_{\Gamma_2}$
be the two points on the boundary that we have selected. At these points, the boundary conditions are $u_{\Gamma_{1}}$ and $u_{\Gamma_{2}}$ respectively. 

We use the following RBF and linear polynomial augmentation to interpolate at $\bx_{\textrm{BP}}$:
\begin{align}
\mathcal{I}_{\textrm{RBF}}u(\bx) = \lambda_{ij} \phi(||\bx-\bx_{ij}||)+\lambda_{\Gamma_1}\phi(||\bx-\bx_{\Gamma_1}||)+\lambda_{\Gamma_2}\phi(||\bx-\bx_{\Gamma_2}||)+\mu_1+\mu_2 x+\mu_3 y,  
\label{eq:rbf}
\end{align}
where $\bx=(x,y)$, $||\cdot||$ is the standard $l_2$ norm and $\phi(\cdot)$ is a radial basis function. The linear polynomial augumentation is required to obtain second order accuracy \cite{barnett2015robust,flyer2016role}. The coefficients $\blambda=(\lambda_{ij},\lambda_{\Gamma_1},\lambda_{\Gamma_2})^T$ and $\bmu=(\mu_1,\mu_2,\mu_3)^T$ are determined by solving the linear system
\be
B\left(\begin{matrix}
\blambda \\
\bmu 
\end{matrix}\right) 
=\left(\begin{matrix}
A & \Pi^T \\
\Pi & 0 
\end{matrix}\right) 
\left(\begin{matrix}
\blambda \\
\bmu 
\end{matrix}\right) 
=\left(\begin{matrix}
\bu \\
{\mathbf{0}}
\end{matrix}\right).
\ee 
Here, $\bu=(u_{ij},u_{\Gamma_1},u_{\Gamma_2})$,
\bd
A = \left(\begin{matrix}
\phi(0) & \phi(||\bx_{ij}-\bx_{\Gamma_1}||) &  \phi(||\bx_{ij}-\bx_{\Gamma_2}||) \\
\phi(||\bx_{\Gamma_1}-\bx_{ij}||) & \phi(0) &  \phi(||\bx_{\Gamma_1}-\bx_{\Gamma_2}||) \\
\phi(||\bx_{\Gamma_2}-\bx_{ij}||) &  \phi(||\bx_{\Gamma_2}-\bx_{\Gamma_1}||) &  \phi(0) 
\end{matrix}\right) 
\quad
\text{and}
\quad
\Pi = \left(\begin{matrix}
1 & 1 & 1\\
x_{ij} & x_{\Gamma_1} & x_{\Gamma_2} \\ 
y_{ij} & y_{\Gamma_1} & y_{\Gamma_2}
\end{matrix}\right).
\ed
The purpose of the last equation $\Pi \blambda = \mathbf{0}$ is to minimize the far-field growth \cite{fornberg2002observations}. 
The  value at the boundary point $\bx_{\textrm{BP}}$ is then
\begin{align}
u_{\textrm{BP}} =\mathcal{I}_{\textrm{RBF}}u(\bx_{\textrm{BP}}) = \left(\bphi_{\textrm{BP}}^T,\bpoly_{\textrm{BP}}^T\right)
B^{-1}\left(\begin{matrix}
\bu \\
{\mathbf{0}}
\end{matrix}\right)
\label{eq:rbf_ubp}
\end{align} 
with $\bphi_{\textrm{BP}}=\left(\phi(||\bx_{\textrm{BP}}-\bx_{ij}||), \phi(||\bx_{\textrm{BP}}-\bx_{\Gamma_1}||), \phi(||\bx_{\textrm{BP}}-\bx_{\Gamma_2}||)\right)^T$ and $\bpoly_{\textrm{BP}}=(1,x_{\textrm{BP}},y_{\textrm{BP}})^T$.  The only interior computational point involved in \eqref{eq:rbf_ubp} is $\bx_{ij}$, hence only the diagonal elements of the discrete Laplacian operator are modified and the symmetry is preserved.
 
We now describe how the two points $\bx_{\Gamma_i}$ ($i=1,2$)  on the boundary are determined (this is also described in Algorithm \ref{alg:rbf_bc_points}.) Our procedure is simple. If the points closest to $\bx_{\textrm{BP}}$ and $\bx_{ij}$ on the boundary are distinguishable, we choose these two points. Otherwise, we choose the point closest to $\bx_{\textrm{BP}}$ as $\bx_{\Gamma_1}$ and pick $\bx_{\Gamma_2}$ such that the angle between $\overrightarrow{\bx_{\textrm{BP}}\bx_{\Gamma_1}}$ and $\overrightarrow{\bx_{\textrm{BP}}\bx_{\Gamma_2}}$ is $\frac{\pi}{4}$. 

\begin{algorithm}[H]
 \caption{Given the computational point $\bx_{ij}$, its neighboring boundary point $\bx_{\textrm{BP}}$ and a pre-selected tolerance $\epsilon$, find the interpolation points $\bx_{\Gamma_1}$ and $\bx_{\Gamma_2}$. \label{alg:rbf_bc_points}} 
 \begin{algorithmic}
\STATE Find the point $\bx_{\Gamma_1}$ and the $\bx_{\Gamma_2}^{(0)}$ such that 
\begin{align}
\bx_{\Gamma_1}^{} = \arg\min_{\bx\in\Gamma_l}||\bx-\bx_{\textrm{BP}}||\quad\text{and}\quad\bx_{\Gamma_2}^{(0)} = \arg\min_{\bx\in\Gamma_l}||\bx-\bx_{ij}||.
\end{align}
\STATE {\bf{If}} 
\begin{align}
|| \bx_{\Gamma_1}-\bx_{\Gamma_2}^{(0)}||>\epsilon h,\label{eq:distinct_bc}
\end{align} 
{\bf{then}} $\bx_{\Gamma_2}=\bx_{\Gamma_2}^{(0)}$.
\STATE
{\bf{Otherwise}}, rotate the line determined by $\bx_{\Gamma_1}$ and $\bx_{\textrm{BP}}$ counterclockwise by $\frac{\pi}{4}$ and find the intersection of this line with $\Gamma$. Choose the intersection point as $\bx_{\Gamma_2}$.
\end{algorithmic}
 \end{algorithm}

\subsection{Choice of interpolation strategy}
We note that the RBF interpolation can always be applied. However, there are two advantages of the line-by-line interpolation over the RBF interpolation. First, it is more straightforward and second, it is possible to prove that it will preserve diagonal dominance. In what follows, if we apply the line-by-line interpolation wherever possible and only use the RBF interpolation when the line-by-line approach breaks down, we say the embedded boundary method is {\bf mixed}. If the RBF interpolation is applied everywhere, we say the embedded boundary method is {\bf RBF-based}. 

Numerically, we have observed that the condition (inequality {\eqref{eq:distinct_bc}) that the two points on the boundary used in the RBF interpolant always are determined to be distinct for $\epsilon=0.025$ if the  mixed EB method is used. However, the inequality \eqref{eq:distinct_bc} may  not always hold if  the RBF-based EB method is used. An often occurring case when this situation may arise is when the closest points to $\bx_{\textrm{BP}}$ and $\bx_{ij}$ are the same point (see \fig{\ref{fig:violate_distinct}}). Then, in the mixed EB method, the line-by-line interpolation is used. We emphasize that although we have not been able to prove that the mixed EB method or the RBF-based method always lead to a SPD system we have not encounter any numerical examples where the SPD property is lost. Below in Section \ref{sec:spd_alg}, we describe a simple algorithm to a-priori check whether the SPD structure is preserved. This algorithm does not require the computation of eigenvalues.

\begin{rem}
To impose the boundary condition through RBF interpolation, we choose one interpolation point as the interior computational point and put the other two interpolation points on the physical boundary. An alternative choice is to put all the three interpolation points on the physical boundary. This alternative choice only modifies the right hand side, and hence always results in a diagonally dominant SPD system. However, in numerical experiments we observe that, without using the information of the interior computational point, this alternative RBF interpolation strategy could sometimes lead to  $O(1)$ errors. Therefore we do not use and do not recommend this alternative strategy.
\end{rem}

 \begin{figure}[htb]
\begin{center}
\setlength{\unitlength}{0.8cm} 
\begin{picture}(5,5) 
\multiput(0,2)(0,2){2}{\line(1,0){6}}
\multiput(0,2)(0,1){3}{\line(1,0){4}}
\multiput(1,1)(1,0){5}{\line(0,1){4}}
\qbezier(2.0,5.0)(0.25,3.0)(2.0,1.0)
\qbezier(2.0,5.0)(0.25,3.0)(2.0,1.0)
\put(2,3){\circle{0.25}}
\put(3,3){\circle*{0.25}}
\put(1.15,3){\circle*{0.25}}

\put(1.75,3.2){$\bx_{\textrm{BP}}$} 
\put(1.05,4.25){$\Gamma$}
\put(2.75,2.5){$\bx_{ij}$} 
\put(0.25,3.3){$\bx_{\Gamma_1}$}
\put(0.5,2.75){\verteq}
\put(0.25,2.6){$\bx_{\Gamma_2}$}
\thicklines
\put(0.5,3){\vector(1,0){5.25}}
\put(5.8,2.5){$\xi$}
\end{picture}

\caption{An example where the condition \eqref{eq:distinct_bc} is violated. To handle this case the rotation step in the  RBF-based algorithm is employed.\label{fig:violate_distinct}}
\end{center}
\end{figure}
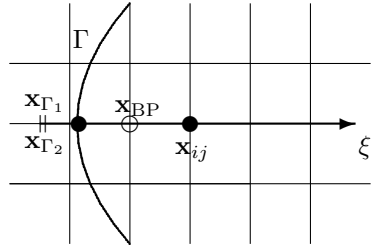

\begin{algorithm}[H]
 \caption{Given a computational domain $\Omega$, a Cartesian mesh, source $f(x,y)$, and boundary conditions, compute the numerical solution to \eqref{eq:peq}. \label{alg:cut_cell} }
 \begin{algorithmic}
\STATE{{\bf{Step 1:}}} Locate the interior computational points and the interior boundary points. Obtain the total number of computational points $N_{\textrm{C}}$.   
\STATE {\bf{Step 2:}} Assemble the discrete Laplacian operator $\mathbf{L}\in \mathbb{R}^{N_{\textrm{C}}\times N_{\textrm{C}}}$ and the right hand side $\mathbf{b}\in \mathbb{R}^{N_{\textrm{C}}}$: apply the difference approximation \eqref{eq:central_difference} at the computational points $\bx_k^{\textrm{C}}=(x_{i^{(k)}},y_{j^{(k)}})$, $1\leq k\leq N_{\textrm{C}}$. If $(x_{i^{(k)}\pm1},y_{j^{(k)}\pm1})$ is a boundary point, temporarily neglect its contribution.
\STATE{\bf{Step 3:}} Impose the boundary condition to correct $\mathbf{L}$ and $\mathbf{b}$:
  \STATE {\bf{If}} $(x_{i^{(k)}-1},y_{j^{(k)}})$ is a boundary point:
 \STATE {\qquad\bf{If}} $i^{(k)}-2\leq 0$ or $(x_{i^{(k)}-2},y_{j^{(k)}})$ is outside $\Omega$, {\bf{then}} apply the line-by-line interpolation to impose boundary condition and correct $\mathbf{L}_{kk}$ and $\mathbf{b}_k$.
\STATE {\qquad\bf{Otherwise}},  apply the RBF interpolation to impose the boundary condition and correct $\mathbf{L}_{kk}$ and $\mathbf{b}_k$. 
  \STATE{\bf{Endif}}
  \STATE If $(x_{i^{(k)}+1},y_{j^{(k)}})$ or $(x_{i^{(k)}},y_{j^{(k)}\pm1})$ is a boundary point, impose the boundary condition similarly.
  \STATE If the RBF interpolation was used, use the algorithm in Section \ref{sec:spd_alg} to check whether the SPD structure is preserved.
\STATE{\bf{Step 4:}} Solve the linear equation $\mathbf{L}\mathbf{u}=\mathbf{b}$ with the conjugate gradient (CG) method and an AMG preconditioner.
\end{algorithmic}
 \end{algorithm}

\subsection{An algorithm to check the SPD structure for constant $\beta$\label{sec:spd_alg}}

It is well known that the conjugate gradient method (CG) with the classical algebraic multigrid preconditioner (AMG) \cite{ruge1987algebraic} is efficient for SPD matrices. The proposed embedded boundary method always results in a symmetric linear system. For many cases, the resulting linear system is also diagonally dominant, which is a sufficient condition for  symmetric positive definiteness. However, for certain geometric configurations the matrix may not be diagonally dominant. For these (rarely occurring) cases we present a simple algorithm to a-priori check whether the matrix is positive definite. 

To derive conditions guaranteeing that the discrete matrix corresponding to our discretization of the Laplacian is SPD, we need a result for the one dimensional three-point central difference discretization for $u_{xx}$ with the same type of boundary modification as in the embedded boundary methods described above.

Consider a discretization along a grid-line in the $x$-direction with $n$ grid points. Assume that the boundary conditions on each side have been imposed by modifying the first and last diagonal element in the matrix (and the right hand side vector). Then the resulting matrix can be written   
\begin{align}
	\mathcal{D}^{(n)}(a,b) =
	\begin{cases}
		\left(\begin{matrix}a\end{matrix}\right), \quad n=1,\\
		\left(\begin{matrix}
			a & -1 \\
			-1 & b
		\end{matrix}\right), \quad n=2,\\
		\left(\begin{matrix}
			a & -1 & 0 & \dots & 0 & 0\\
			-1 &  2 &-1& \dots  & 0 & 0\\
			0 & -1 & 2 & \dots  & 0 & 0\\
			\vdots & &\vdots & \vdots & \vdots & \vdots\\
			0 & 0 & 0 & \dots & 2 & -1\\
			0 & 0 & 0 & \dots &-1  & b
		\end{matrix}\right) \in \mathbb{R}^{n\times n}, \quad n\geq 3.
	\end{cases} \label{eq:tridiag}
\end{align}

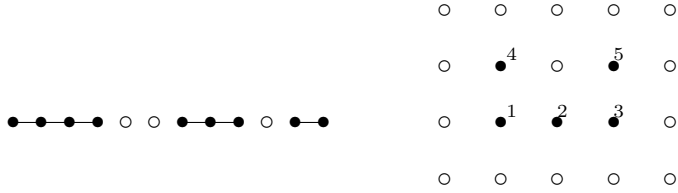
\begin{figure}[!htbp]
	\centering
	\setlength{\unitlength}{1.5cm} 
	\begin{picture}(3,1) 
		\multiput(0,1)(0,1){1}{\line(1,0){0.75}}
		\multiput(1.5,1)(0,1){1}{\line(1,0){0.5}}
		\multiput(2.5,1)(0,1){1}{\line(1,0){0.25}}
		\put(0,1){\circle*{0.1}}
		\put(0.25,1){\circle*{0.1}}
		\put(0.5,1){\circle*{0.1}}
		\put(0.75,1){\circle*{0.1}}
		\put(1.0,1){\circle{0.1}}
		\put(1.25,1){\circle{0.1}}
		\put(1.5,1){\circle*{0.1}}
		\put(1.75,1){\circle*{0.1}}
		\put(2.0,1){\circle*{0.1}}
		\put(2.25,1){\circle{0.1}}
		\put(2.5,1){\circle*{0.1}}
		\put(2.75,1){\circle*{0.1}}
	\end{picture}
	\hspace{1cm}
	\begin{picture}(3,1) 
		\put(0,0.5){\circle{0.1}}
		\put(0.5,0.5){\circle{0.1}}
		\put(1.0,0.5){\circle{0.1}}
		\put(1.5,0.5){\circle{0.1}}
		\put(2.0,0.5){\circle{0.1}}
		\put(0,1.0){\circle{0.1}}
		\put(0.5,1.0){\circle*{0.1}}
		\put(1.0,1.0){\circle*{0.1}}
		\put(1.5,1.0){\circle*{0.1}}
		\put(2.0,1.0){\circle{0.1}}
		\put(0.55,1.05){\scriptsize$1$}
		\put(1.0,1.05){\scriptsize$2$}
		\put(1.5,1.05){\scriptsize$3$}
		\put(0,1.5){\circle{0.1}}
		\put(0.5,1.5){\circle*{0.1}}
		\put(1.0,1.5){\circle{0.1}}
		\put(1.5,1.5){\circle*{0.1}}
		\put(2.0,1.5){\circle{0.1}}
		\put(0.55,1.55){\scriptsize$4$}
		\put(1.5,1.55){\scriptsize$5$}
		\put(0,2){\circle{0.1}}
		\put(0.5,2){\circle{0.1}}
		\put(1.0,2){\circle{0.1}}
		\put(1.5,2){\circle{0.1}}
		\put(2.0,2){\circle{0.1}}
	\end{picture}
	\caption{Left: computational points and non-computational points along a grid line. Here the ordering of the unknowns would be increasing from left to right.  Right: lexicographic ordering of the computational points results in a block diagonal discretization matrix (here with block sizes 3,1,1. Solid points: computational points. Empty points: non-computational points. \label{fig:spd_geometry}}
\end{figure}

The following Lemma (whose proof is given in Appendix \ref{appendix:proof}) gives conditions on $a$ and $b$ which  guarantee that $\mathcal{D}^{(n)}(a,b)$ is SPD.
\begin{lem}\label{lem:1d_spd}
The matrix $\mathcal{D}^{(n)}(a,b)$ is symmetric positive definite, if one of the following conditions is satisfied:
\begin{enumerate}
\item 
 \begin{align}
n=1,~{\rm and }~a>0.
\end{align}
\item  
\begin{align}
n=2,~ a>0,~{\rm and}~ab>1.
\end{align}
\item 
 \begin{align}
n\geq 3, \quad a>\frac{n-2}{n-1},\quad b>\frac{n-2}{n-1}\quad\text{\rm and}\quad a> \frac{(n-2)b-(n-3)}{(n-1)b-(n-2)}.
\end{align}
\end{enumerate}
\end{lem}

As a result of Lemma \ref{lem:1d_spd},  the following corollary gives us two conditions that are more straightforward to check.  
\begin{cor}\label{cor:spd}
The matrix $\mathcal{D}^{(n)}(a,b)$ is symmetric positive definite if $a>1$ and $b>1$.
\end{cor}

In general there may be holes in the computational domain and the discretization along a grid line is then divided into $S$ segments (see the example in the left picture in Figure \ref{fig:spd_geometry}). Suppose that the $k$-th segment contains $n_k$ computational points ordered from left to right, then the matrices corresponding to the discretization on each grid line segment can be arranged in the following block diagonal form
\begin{align}
\mathcal{D}^{{\rm line}}=\left(\begin{matrix}
\mathcal{D}^{(n_1)}(a_1,b_1) & & & \\
& \mathcal{D}^{(n_2)}(a_2,b_2) & & \\
& & \ddots  &\\
& & & \mathcal{D}^{(n_S)}(a_S,b_S) 
\end{matrix}\right).\label{eq:general_1d_matrix}
\end{align}
The matrix $\mathcal{D}^{{\rm line}}$ thus contains the discretization of the second derivative along a single, say, horizontal grid line. As a direct result of Lemma \ref{lem:1d_spd} and the block structure of \eqref{eq:general_1d_matrix}, we have the following Theorem.
\begin{thm}\label{thm:1d_spd}
The matrix $\mathcal{D}^{\rm{line}}$, defined in \eqref{eq:general_1d_matrix}, is symmetric positive definite, if for each of the blocks, $\mathcal{D}^{(n_k)}(a_k,b_k)$, the numbers $a_k,b_k,n_k$ satisfy one of the conditions in Lemma \ref{lem:1d_spd}.
\end{thm}

In higher dimensions there will be many grid lines and many segments but, assuming lexicographic ordering (see the right picture in Figure \ref{fig:spd_geometry}), the discretization matrix for the second derivative in $x$ will still be block diagonal with each block being a tridiagonal matrix on the form (\ref{eq:tridiag}). Let this ``two dimensional'' block diagonal matrix with tridiagonal blocks be called $\mathcal{D}_{xx}$. Similarly let the matrix $\mathcal{D}_{yy}$ be a block diagonal matrix with tridiagonal blocks corresponding to the discretization of the second derivative in $y$ along all grid lines but now with an ordering of the degrees of freedom that is fast in the $y$-index. Finally let $P$ be the permutation matrix that converts between the fast-in-$x$ (lexicographic) and fast-in-$y$ ordering. Then, using the lexicographic ordering the approximation of the Laplacian is $-\mathcal{L}=\mathcal{D}_{xx}+P^T\mathcal{D}_{yy}P$ and we have the following theorem.

\begin{thm}\label{thm:spd_summation}
Consider a two dimensional embedded boundary discretization on a grid with $N_c$ computational points and resulting in block diagonal matrices $\mathcal{D}_{xx}\in\mathbb{R}^{N_c\times N_c}$ and $\mathcal{D}_{yy}\in\mathbb{R}^{N_c\times N_c}$ in the fast-in-$x$  and fast-in-$y$ orderings, respectively. Assume that all of the one dimensional one segment discretization matrices (in both the $x$ and $y$ direction) satisfy the conditions in Lemma \ref{lem:1d_spd}, then the matrix  $-\mathcal{L}=\mathcal{D}_{xx}+P^T\mathcal{D}_{yy}P$ is symmetric positive definite.   
\end{thm}
\begin{proof}
The matrix is manifestly symmetric. Let $v\in \mathbb{R}^{N_C}$ be an arbitrary non-zero vector. With $\mathcal{D}_{xx}$ and $\mathcal{D}_{yy}$ being symmetric positive definite, $ v^T \mathcal{D}_{xx} v>0$ and $v^T P^T \mathcal{D}_{yy}Pv>0$ and so  
\begin{align}
-v^T \mathcal{L} v = v^T \mathcal{D}_{xx} v+ v^T P^T\mathcal{D}_{yy}Pv>0+0=0.
\end{align}
\end{proof}

At the implementation level, when we check the conditions in Lemma \ref{lem:1d_spd}, the width of each 1D segments can be found
based on the mask matrix, and the values of $a_k$ and $b_k$ can be computed based on the boundary corrections along the horizontal (or vertical) direction.

\section{Numerical results}\label{sec:num}
We now demonstrate performance of our method through series of numerical examples including Poisson's equation, the Helmholtz equation, the heat equation and the wave equation. Throughout the $l_2$ error and $l_{\infty}$ error are computed by
\begin{align}
\mathcal{E}_{l_2} = \sqrt{\sum_{i,j} h^2\left(u_{ij}-u_{\textrm{exact}}(x_i,y_j)\right)^2},
\;
\mathcal{E}_{l_\infty} = \max_{i,j}\left|u_{ij}-u_{\textrm{exact}}(x_i,y_j)\right|.
\end{align}
Throughout this section, DOF stands for the degrees of freedom.
In the RBF interpolation, we choose the polyharmonic spline $\phi(r)=r^3$ as the radial basis function. For the Poisson and the heat equation with implicit time stepper, the linear solver is chosen as the conjugate gradient (CG) method  with a classic algebraic multigrid (AMG) preconditioner \cite{ruge1987algebraic}. Both the $V$-cycle and $W$-cycle are considered in the AMG  method. The iterative solver is considered to have converged when the relative residual smaller than $10^{-12}$. In all of our numerical experiments, we observe that the resulting discrete Laplacian operator is always SPD. 

Our code is implemented in Julia, and we use the AMG preconditioner and CG solver from the open source packages \verb+AlgebraicMultigrid.jl+ and \verb+IterativeSolvers.jl+.

\subsection{Poisson's equation in a non-convex geometry}
\begin{figure}[htb]
\centering
\includegraphics[width=0.45\textwidth,trim={0.2cm 0.2cm 1.75cm 0.2cm},clip]{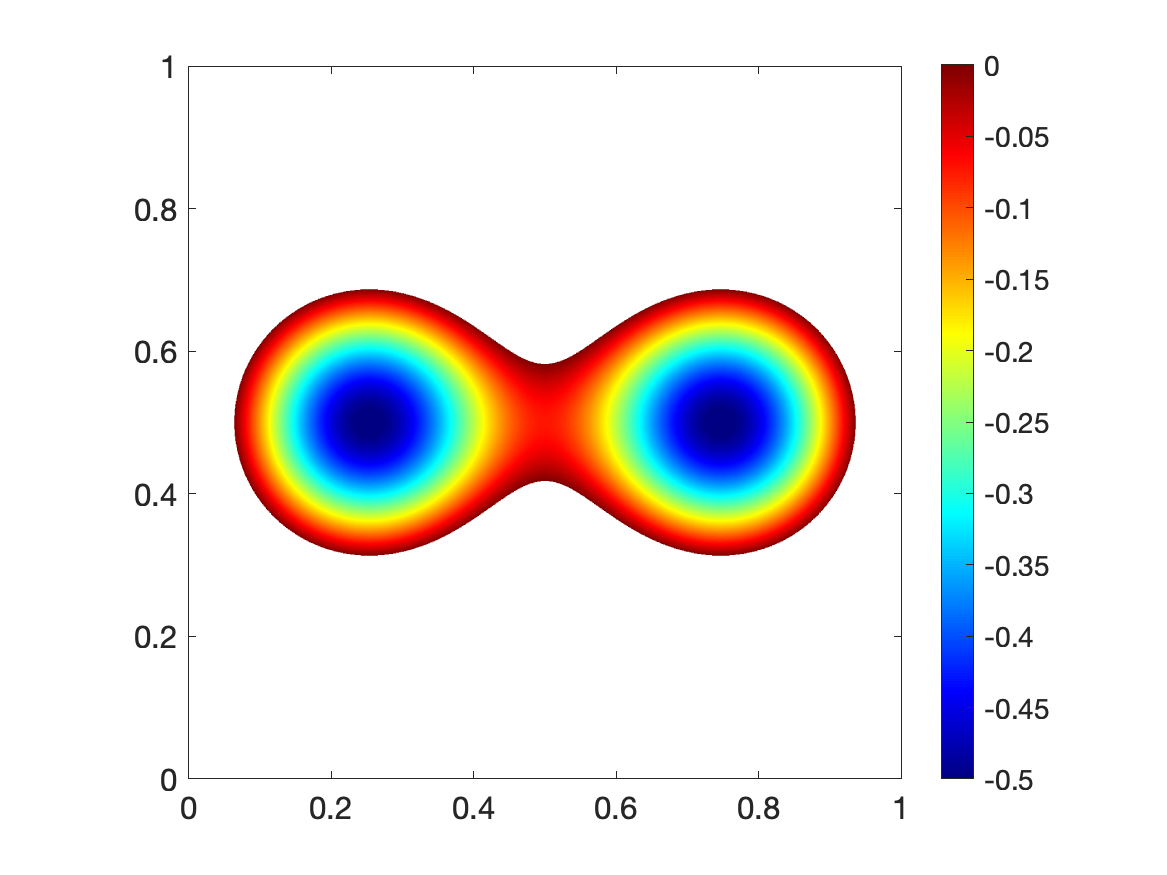}
\includegraphics[width=0.45\textwidth,trim={0.2cm 0.2cm 1.75cm 0.2cm},clip]{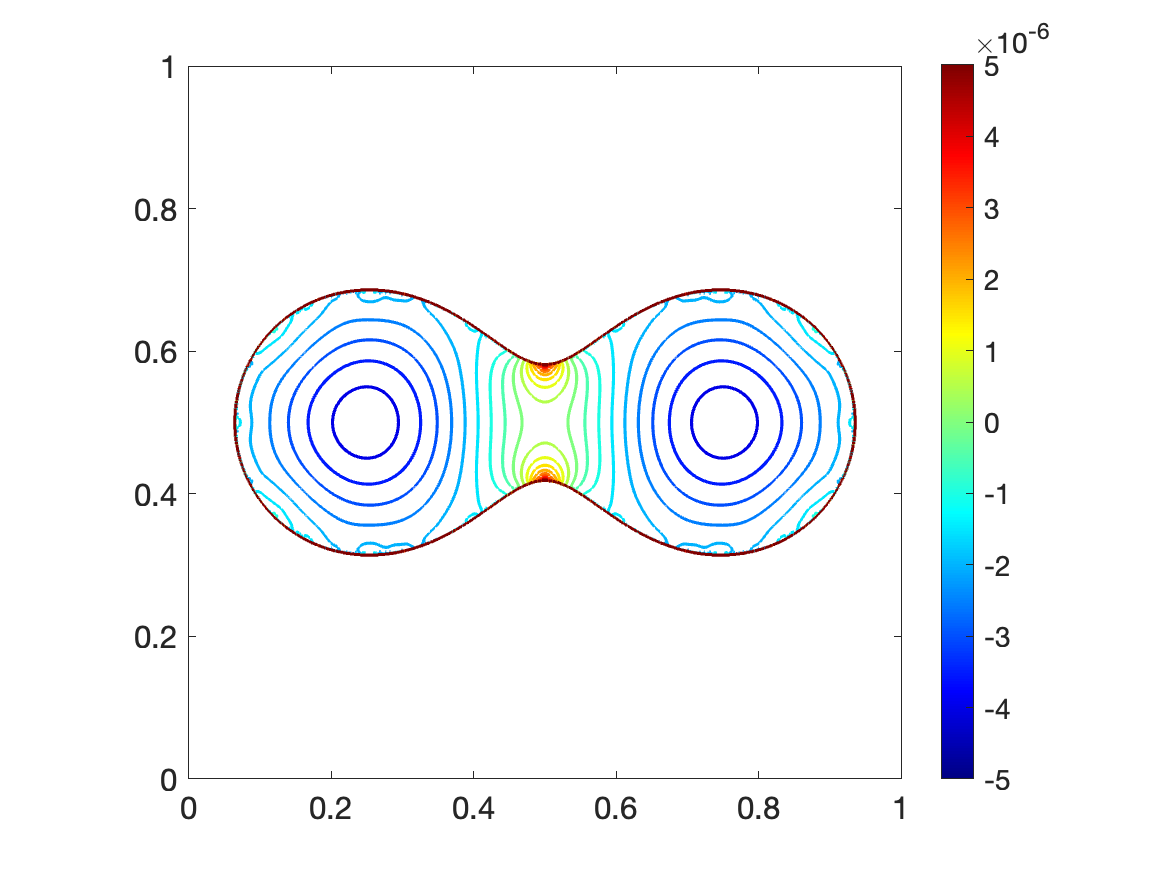}
\caption{Poisson's equation in two dimensions, $\nabla\cdot(\beta\nabla u) = f$, with Dirichlet boundary conditions on a glass-shaped domain. Left: Numerical solution with $N=1280$. Right: The error between the numerical and the exact solution. \label{fig:glass1}}
 \end{figure}
We consider a glass-shaped and non-convex geometry determined by the level-set function
\begin{align}
\psi(x,y) =0.5-e^{-20( (x-0.25)^2+(y-0.5)^2)} -e^{-20( (x-0.75)^2+(y-0.5)^2)}.
\end{align}
The Dirichlet boundary condition and the source function are chosen such that $\psi(x,y)$ is an exact solution with $\beta(x,y)=-8$. An $(N+1)\times (N+1)$ uniform mesh partitioning $[0,1]\times[0,1]$ is used. With this non-convex geometry, the RBF interpolation will be activated in the mixed EB method.
\begin{figure}[htb]
\centering
\includegraphics[width=0.32\textwidth,trim={0.2cm 0.2cm 1.9cm 0.2cm},clip]{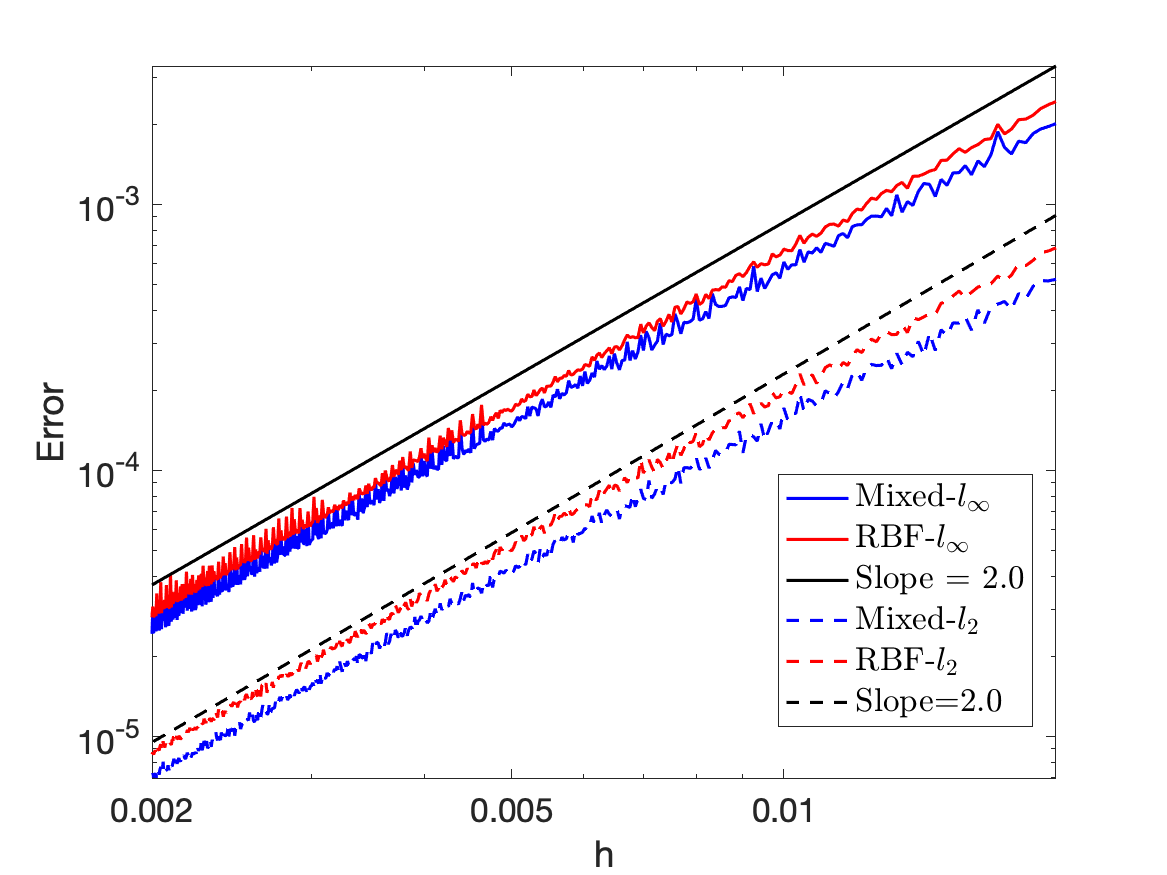}
\includegraphics[width=0.32\textwidth,trim={0.2cm 0.2cm 1.9cm 0.2cm},clip]{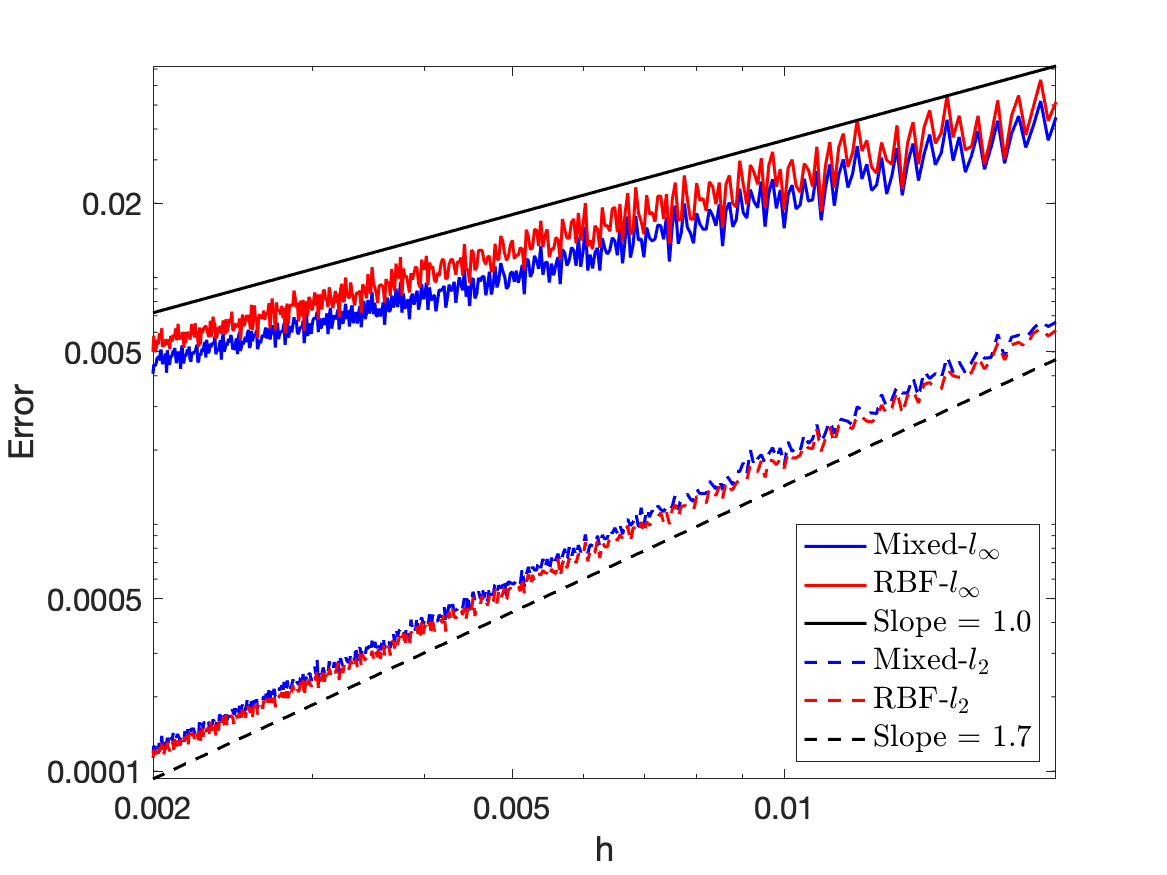}
\includegraphics[width=0.32\textwidth,trim={0.2cm 0.2cm 1.9cm 0.2cm},clip]{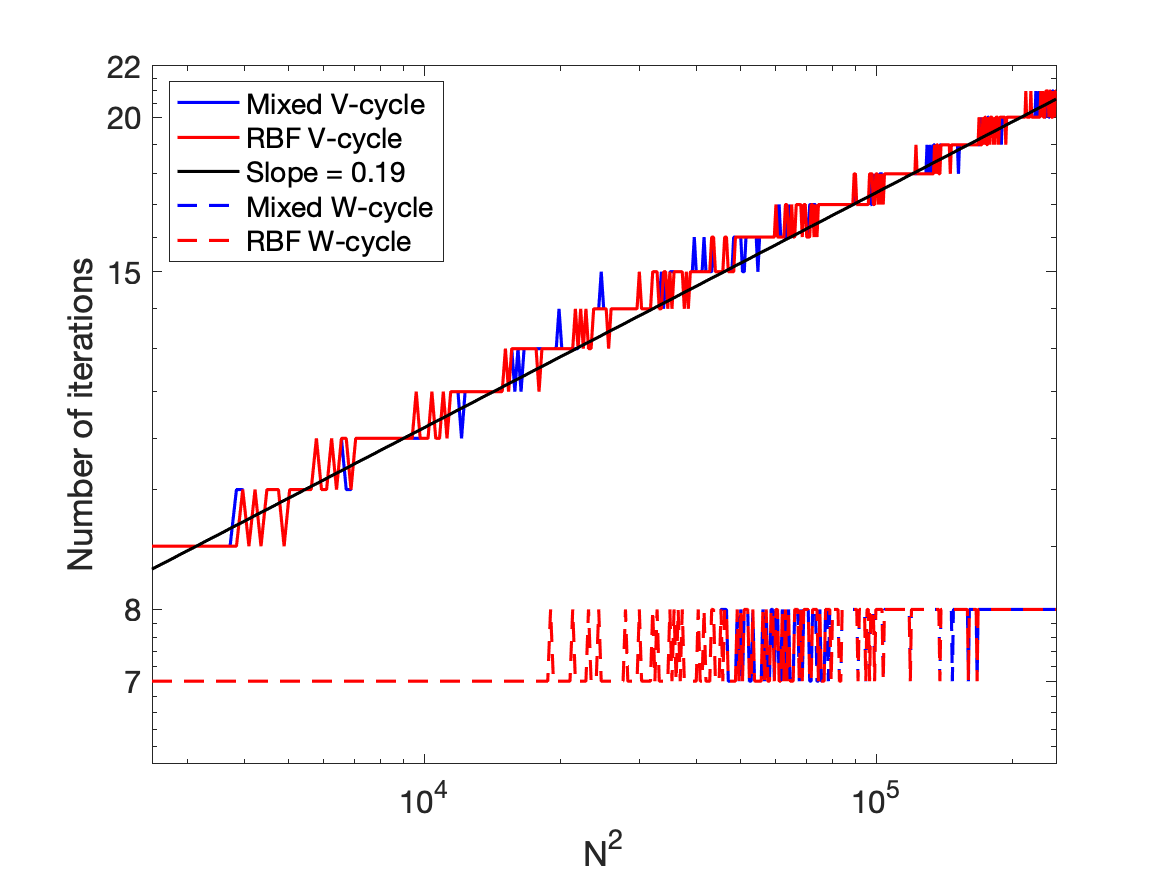}
\caption{Convergence check for the AMG preconditioned embedded boundary method on the glass-shapes domain. Left: Numerical errors correspond to different grid size $h$. Middle: Numerical errors of the gradient for different $h$. Right: The total number of iterations for convergence for different $N^2$. \label{fig:glass2}}
 \end{figure}

The numerical solution and error for $N=1280$ are presented in \fig{\ref{fig:glass1}}. In \fig{\ref{fig:glass2}} the errors for $N=50$ to $N=500$ are  also displayed.  For both of the mixed and the RBF-based method, the $l_2$ and $l_\infty$ error converges as $O(h^2)$. 
The error of the gradient $\nabla u$ converges as $O(h^{1.7})$ in $l_2$ and $O(h)$ in $l_\infty$. The total number of iterations for convergence for the $W$-cycle are independent of the size of the problem and are always  $7$ or $8$. The number of iterations for the $V$-cycle scale as $O(DOF^{0.19})$. 

 \begin{figure}[htb]
 \centering
\includegraphics[width=0.45\textwidth,trim={0.2cm 0.2cm 1.4cm 0.2cm},clip]{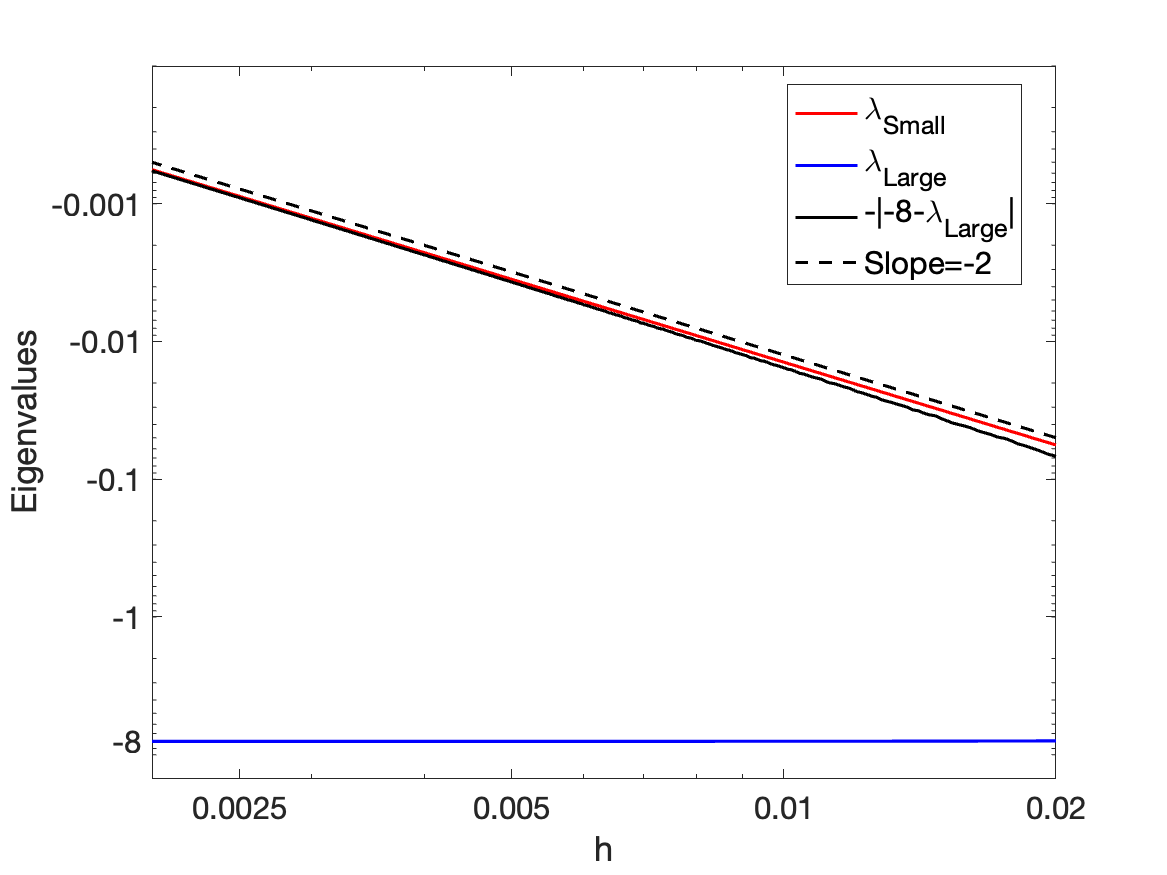}
\includegraphics[width=0.45\textwidth,trim={0.2cm 0.2cm 1.4cm 0.2cm},clip]{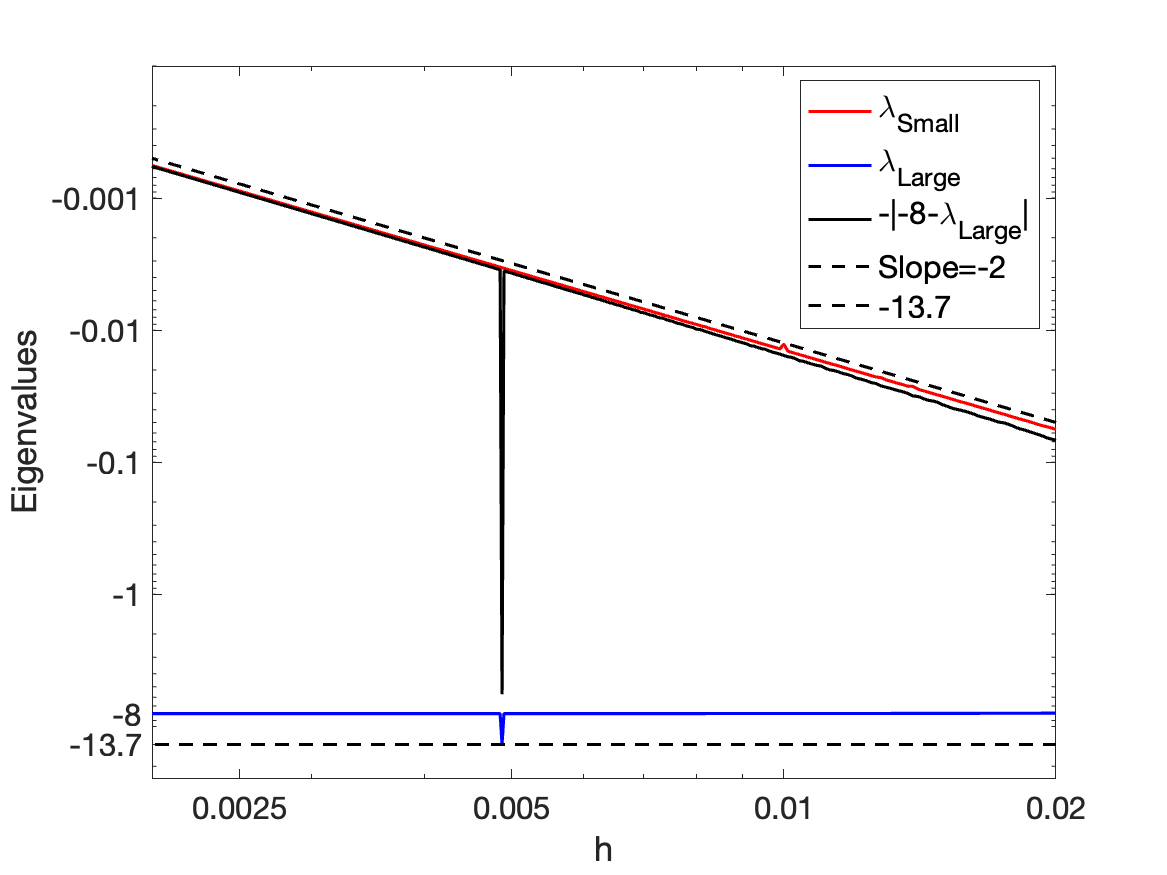}
\caption{Glass shaped geometry: the maximum and minimum eigenvalues of the discrete Laplace operator as a function of $h$. Left: the embedded boundary method using mixed interpolation. Right: the embedded boundary method using  RBF-based interpolation. \label{fig:glass_eig}}
 \end{figure}

We also compute the eigenvalues with the smallest and largest magnitude, $\lambda_{\textrm{Small}}$ and $\lambda_{\textrm{Large}}$, for $50\leq N\leq 500$, and plot them in \fig{\ref{fig:glass_eig}}. As can be seen they are all negative and as expected $\lambda_{\textrm{Small}}$ scales approximately as $O(h^{-2})$ while $\lambda_{\textrm{Large}}$ converges to $-8$ as the mesh is refined.

\subsection{Poisson's equation on a tilted square}
\begin{figure}[htb]
\centering
\includegraphics[width=0.45\textwidth,trim={0.2cm 0.2cm 1.75cm 0.2cm},clip]{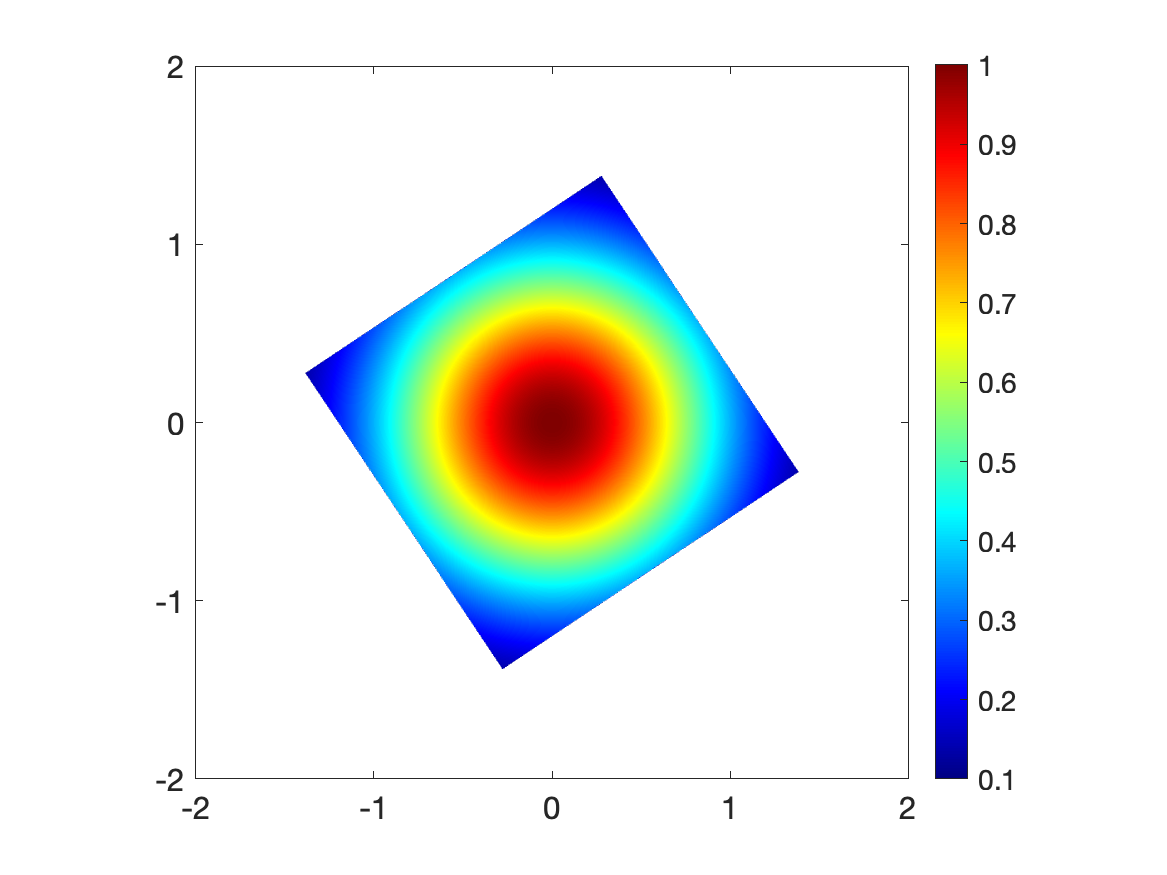}
\includegraphics[width=0.45\textwidth,trim={0.2cm 0.2cm 1.75cm 0.2cm},clip]{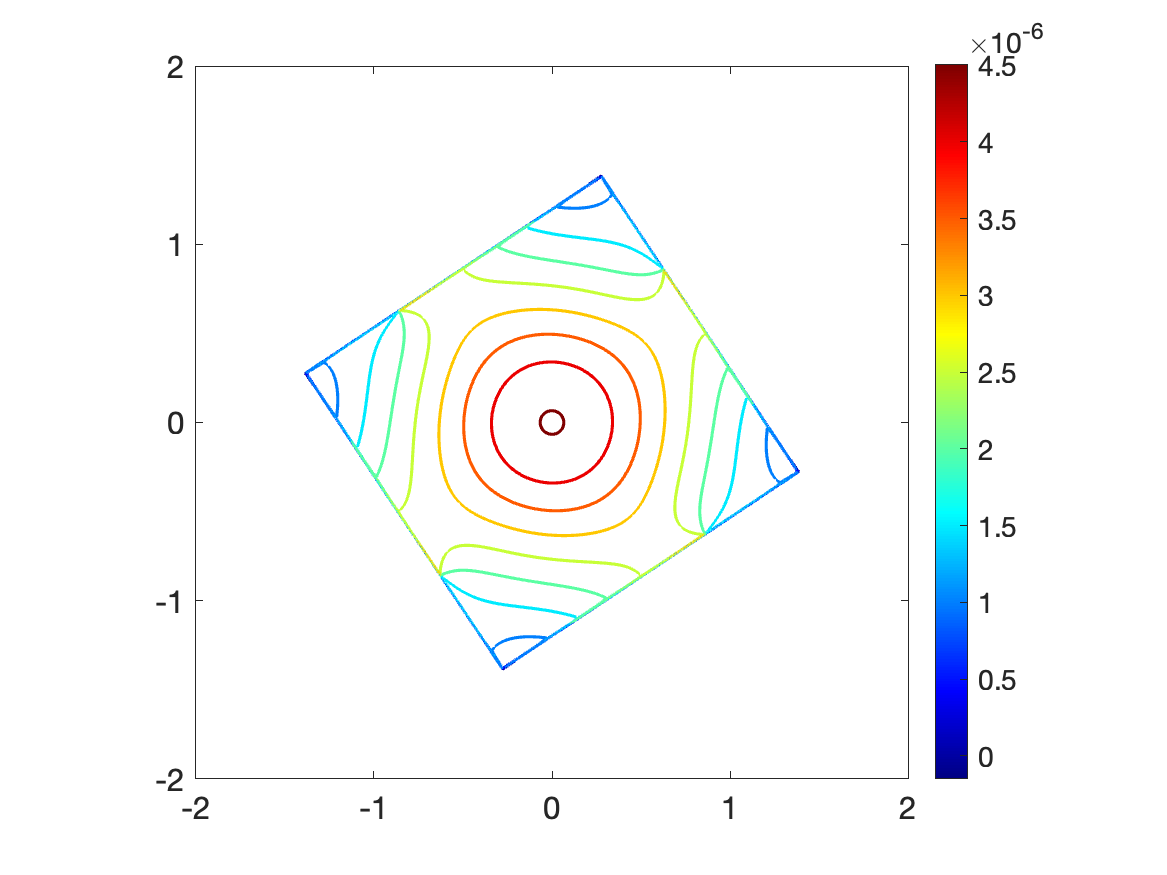}
\caption{Poisson's equation in two dimensions, $\nabla\cdot(\beta\nabla u) = f$, with Dirichlet boundary conditions on a tilted-square domain. Left: Numerical solution with $N=1280$. Right: The  contours of the error between the  numerical and the exact solution.  \label{fig:tilted_square_sol}}
\end{figure}
\begin{figure}[h!]
\centering
\includegraphics[width=0.32\textwidth,trim={0.2cm 0.2cm 1.9cm 0.2cm},clip]{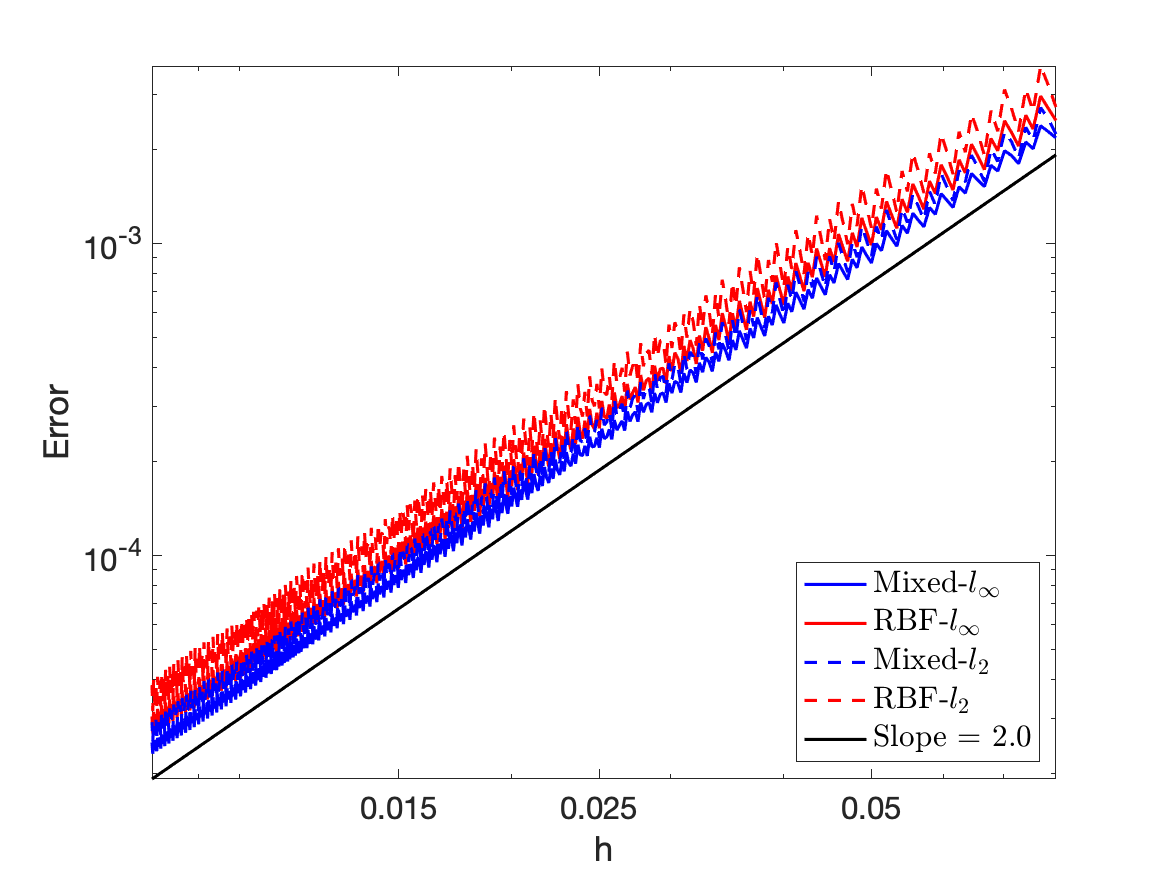}
\includegraphics[width=0.32\textwidth,trim={0.2cm 0.2cm 1.9cm 0.2cm},clip]{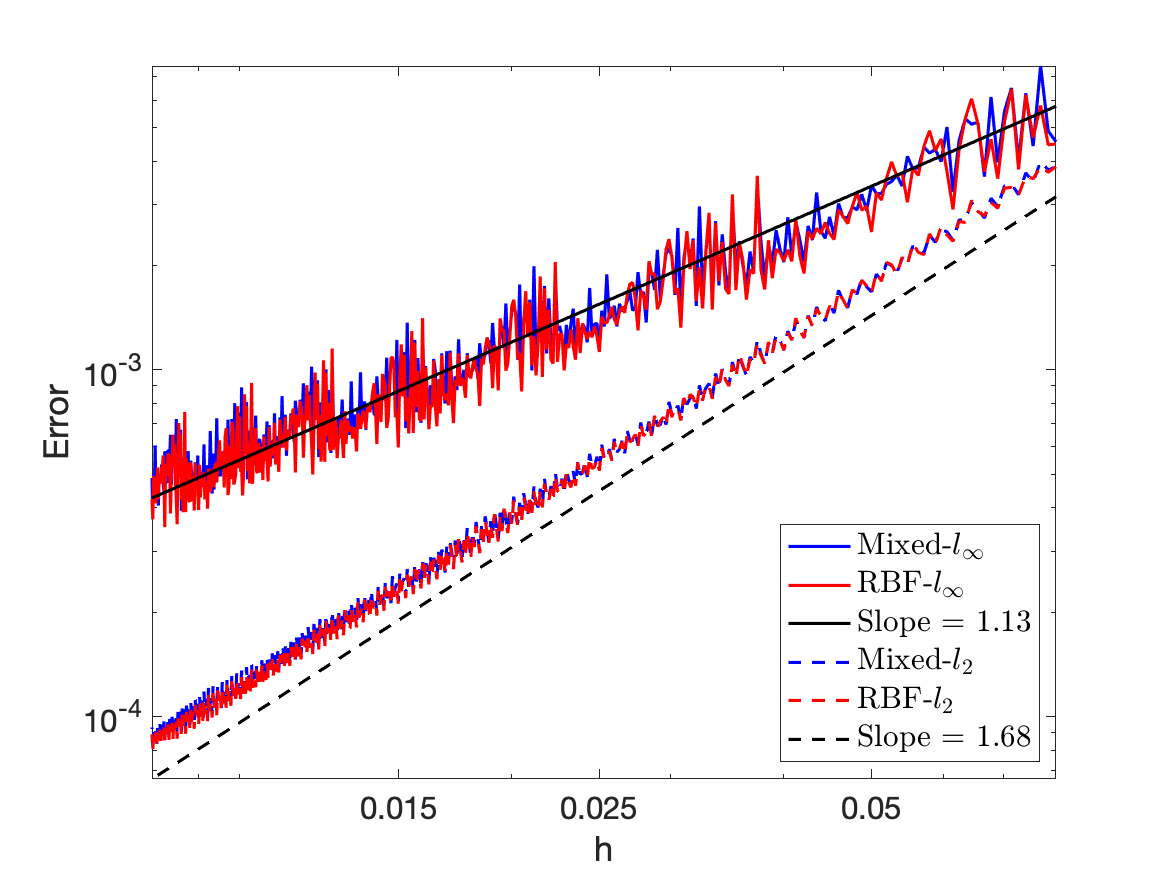}
\includegraphics[width=0.32\textwidth,trim={0.2cm 0.2cm 1.9cm 0.2cm},clip]{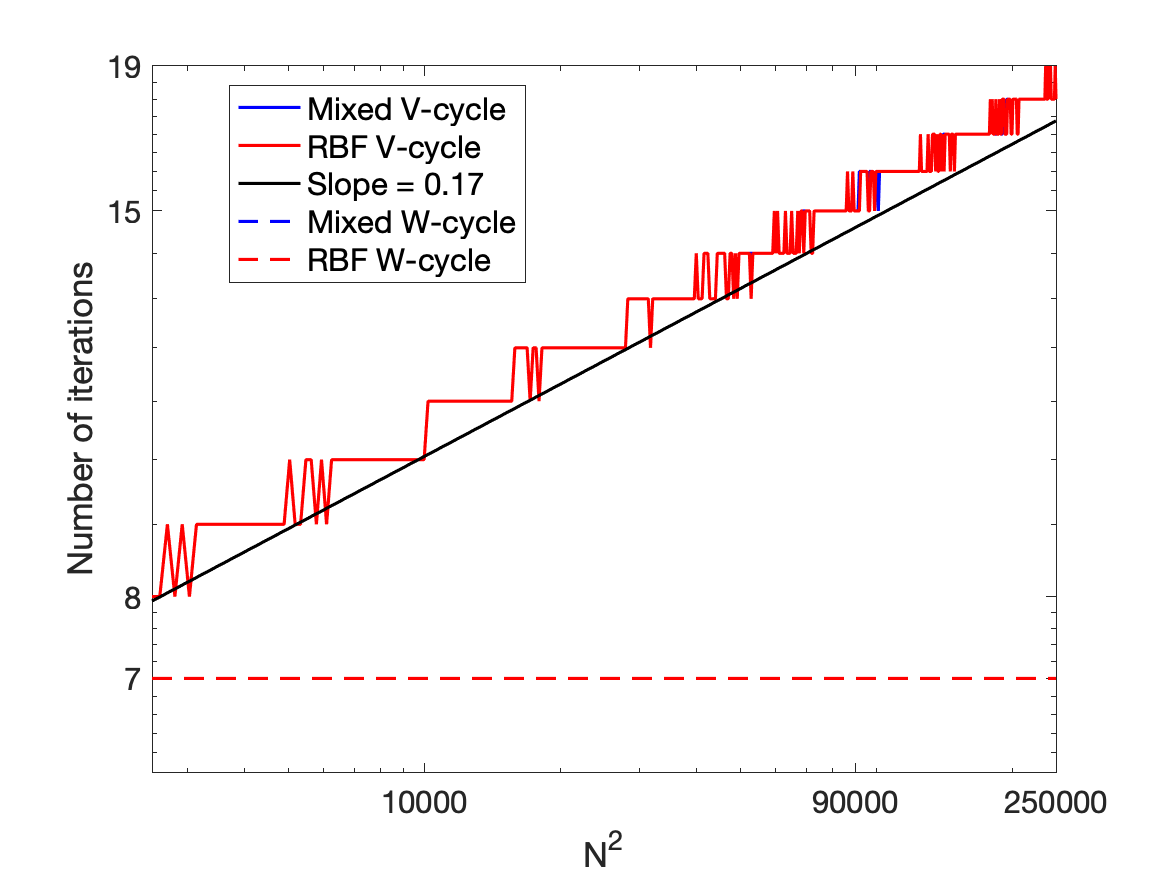}
\caption{Convergence check for the AMG preconditioned  embedded boundary method on a tilted-square domain. Left: Numerical errors as a function of $h$. Middle: Numerical errors of the gradient as a function of $h$. Right: The total number of iterations for convergence for different $N^2$.\label{fig:tilted_square}}
\end{figure}
We take the following example from \cite{ng2009guidelines}. The computational domain is $[-3,3] \times [-3,3]$, and $\Omega$ is a tilted square determined by the level set function 
$$\psi(x,y)=\max\Big(\max\big(\mid(\hat{x}-p_x)-(\hat{y}-p_y)\mid-1, \mid(\hat{x}-p_x)+(\hat{y}-p_y)\mid-1,  \mid(\hat{y}-p_y)-(\hat{x}-p_x)\mid-1\big)\Big),$$ 
where $\hat{x}(x,y)=x\cos(\theta\pi)-y \sin(\theta\pi)$ and $\hat{y}(x,y)=x\sin(\theta\pi)+y \cos(\theta\pi)$. We take $p_x=0.691$ and $p_y=0.357$ so that the center of the tilted square $(p_x,p_y)$ does not fall exactly on a grid point.   We take $\theta=0.313$ so that  the tilted square is not symmetric in the $x$ or $y$ direction. The boundary of the tilted square $\Omega$ has a kink. The exact solution for this example is $u(x,y)=e^{-x^2-y^2}$ and $\beta(x,y)=1.0$. A $(N+1)\times (N+1)$ uniform mesh is used. 

The numerical solution and error for $N=1280$ are presented in \fig{\ref{fig:tilted_square_sol}}. Sweeping from $N=50$ to $N=500$, numerical results are presented in \fig{\ref{fig:tilted_square}}. For both  the mixed and the RBF-based method, we observe second order accuracy for the solution in both $l_\infty$ and $l_2$. The error in the gradient $\nabla u$ scales as $O(h^{1.13})$ in $l_\infty$ and $O(h^{1.68})$ in $l_2$. As in the example above, the total number of iteration for convergence stays fixed at  $7$ when the $W$-cycle is used and scales as $O(DOF^{0.17})$ when the  $V$-cycle is used.

This geometry can result in stencils that are no longer diagonally dominant. Out of the $451$ considered values of $N$, $25$ of them result in a linear system that is not diagonal dominant. For all the system sizes the criteria to check the SPD property from Section \ref{sec:spd_alg} guarantees that the matrix is SPD.


\subsection{Geometry determined by parametric curves}
In our previous examples, $\Omega$ is determined by a level set function. Here we consider a case when it is determined by a parametric curve:
\begin{align}
\Gamma = \{(x,y) = \left(x(\theta),y(\theta)\right),\;\theta\in[\theta_L,\theta_R]\}.
\end{align}
We define the level set function $\psi(x,y)$ as the signed distance function $\psi_{\textrm{SD}}(x,y)$. Let  the closet point on the boundary interface $\Gamma$ to $\bx=(x,y)$ be $\bx_{\mathbf{n}}=(x_\mathbf{n},y_\mathbf{n})$. The amplitude of $\psi_{\textrm{SD}}(x,y)$ is $||\bx-\bx_\mathbf{n}||$. The sign of $\psi_{\textrm{SD}}(x,y)$ is determined by the cross-product of the normal vector $\mathbf{n}=\overrightarrow{\bx\bx_\mathbf{n}}$ and the tangent vector $\boldsymbol{\tau}$ of $\Gamma$ passing $\bx_\mathbf{n}$, and $\psi_{\textrm{SD}}(x,y)$ is negative for $\bx=(x,y)$ inside $\Omega$. 

To  find the closet point  to $\bx=(x,y)$ on the boundary $\Gamma$, a good initial guess is needed. We uniformly partition $[\theta_L,\theta_R]$ with $N_{\textrm{guess}}$ points, and use the point corresponding to 
$\theta_{\textrm{guess}}=\arg\min_{1\leq k\leq N_\textrm{guess}}\{||(x,y)-(x(\theta_k),y(\theta_k))||\}$ as the initial guess.

When the line-by-line interpolation is applied, we need to approximate the horizontal or vertical distance from the interior boundary points to the boundary interface $\Gamma$. One can compute this distance exactly, but  here we re-use the infrastructure for the levelset description of the geometry and instead use a  second order approximation of the desired distance. 
Take the case in \fig{\ref{fig:distancebd}} as an example, $\bx_{i,j}$ is a computational point, $\bx_{i-1,j}$ is a boundary point and $\bx_{i-2,j}$ is outside $\Omega$. $\bx_\Gamma$ lies on the boundary interface $\Gamma$.  Following  \cite{chen1997simple}, $||\bx_{i-1,j}-\bx_{\Gamma}||$ can be  approximated as:
\begin{align}
	||\bx_{i-1,j}-\bx_{\Gamma}|| = \frac{\psi_{\textrm{SD}}(x_{i-1},y_j)}{\psi_{\textrm{SD}}(x_{i-1},y_j)-\psi_{\textrm{SD}}(x_{i-2},y_j)}h+O(h^2).
\end{align}

\begin{figure}[htb]
	\begin{center}
		\setlength{\unitlength}{0.8cm} 
			\begin{picture}(7,7) 
			\multiput(0,1)(0,1){3}{\line(1,0){7}}
			\multiput(0,4)(0,1){3}{\line(1,0){7}}
			\multiput(1,0)(1,0){6}{\line(0,1){7}}
			\qbezier(5.5,6.5)(-2.5,1.5)(4.8,0.2)
			\put(0.85,3){\framebox{}}
			\put(2,3){\circle{0.25}}
			\put(3,3){\circle*{0.25}}

			\put(3.2,5.3){$\Gamma_l$}
			\put(1.8,3.3){$\bx_{i-1,j}$} 
			\put(2.75,2.5){$\bx_{i,j}$} 
			\put(0.6,2.5){$\bx_{i-2,j}$} 

			\put(2.45,1.5){$\bx_{\Gamma}$} 
			\put(2.75,1.75){\vector(-1,1){1.2}}
		\end{picture}
		\caption{ \label{fig:distancebd}Two grid points $\bx_{i-1,j}$ and $\bx_{i-2,j}$  which border $\bx_{\Gamma}$. An approximation of the distance between a {boundary point $\bx_{i-1,j}$ } and the intersection point $\bx_{\Gamma}$ of the interface with the line $y_j$ by the linear interpolation of the level set function.}
	\end{center}
\end{figure}
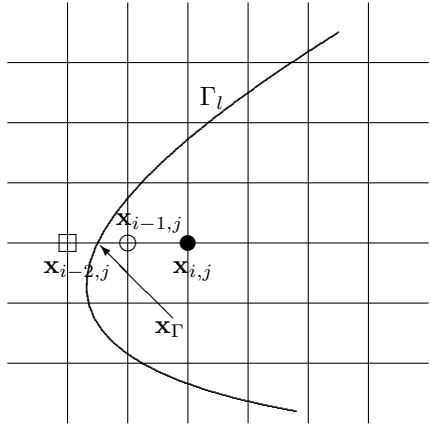

\subsubsection{Poisson's equation on a bone-shaped geometry\label{sec:bone_shape}}
\begin{figure}[htb]
\centering
\includegraphics[width=0.45\textwidth,trim={0.2cm 0.2cm 1.75cm 0.2cm},clip]{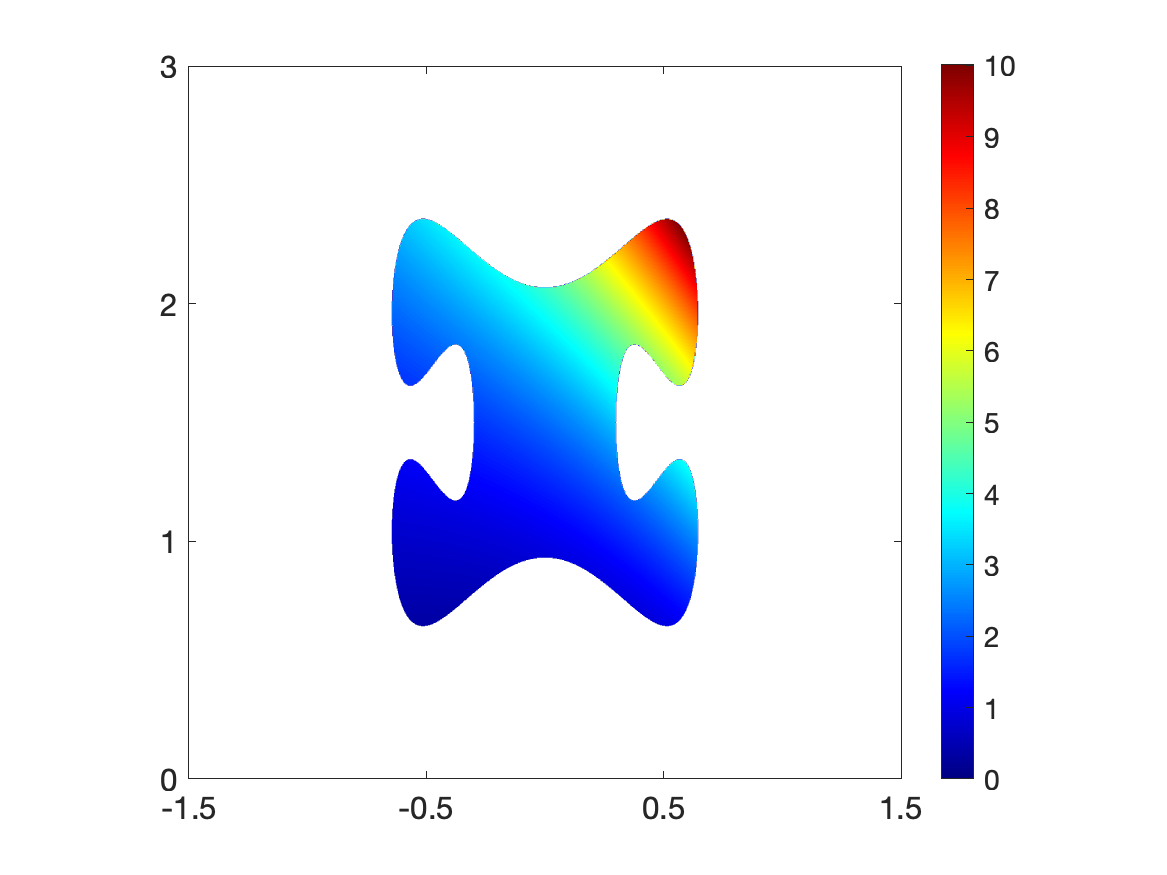}
\includegraphics[width=0.45\textwidth,trim={0.2cm 0.2cm 1.75cm 0.2cm},clip]{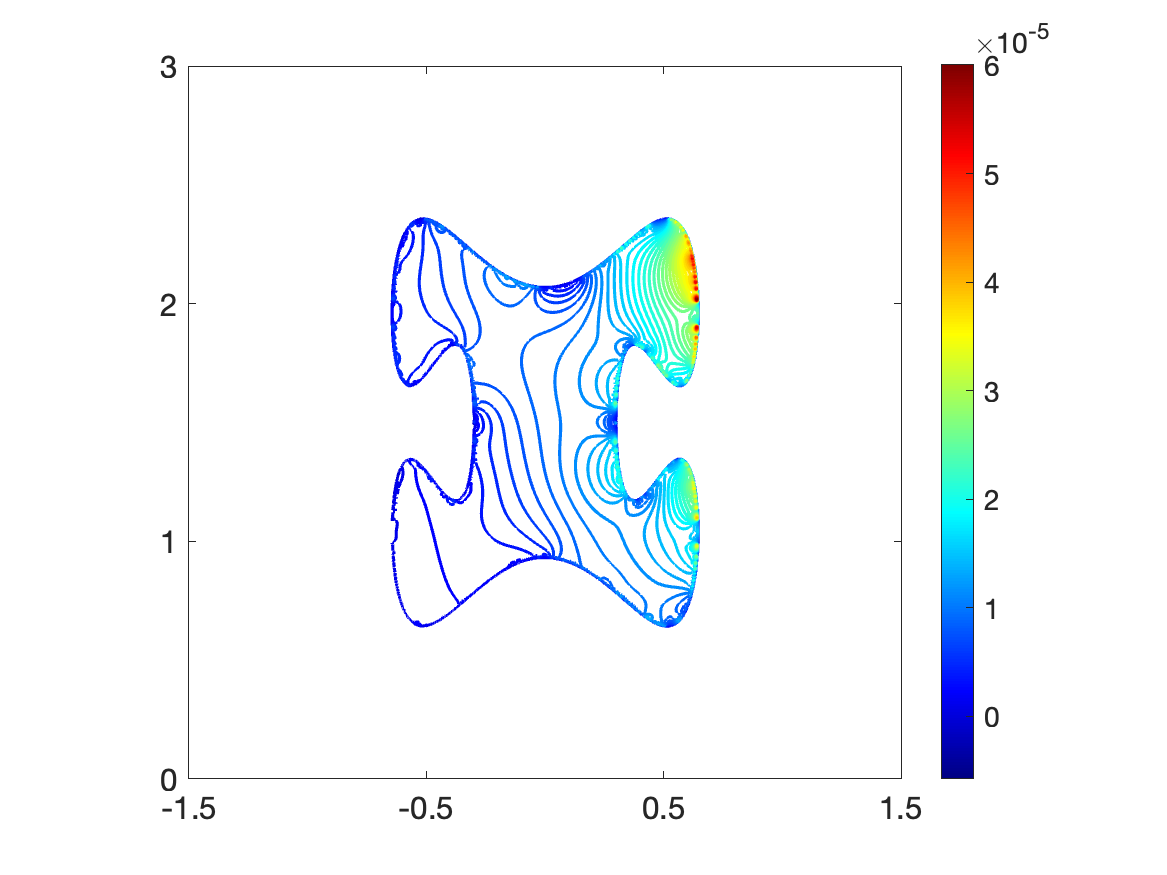}
\caption{Poisson's equation in two dimensions, $\nabla\cdot(\beta\nabla u) = f$, with Dirichlet boundary conditions on a bone-shape domain. Left: Numerical solution with $N=1280$. Right: The  contours of the error between the  numerical and the exact solution. \label{fig:bone_sol}} 
\end{figure}

We take the example from \cite{gibou2002second,li1998fast} and solve  Poisson's equation on a bone-shaped irregular domain in  the computational domain $[-1.5,1.5] \times [0,3]$. The exact solution is  $u=e^x(x^2\sin(y)+y^2)$ and $\beta=2+\sin(xy)$. The boundary $\Gamma$ is parameterized by 
$$\Big(x(\theta),y(\theta)\Big)=\Big(0.6\cos(\theta)-0.3\cos(3\theta),1.5+0.7\sin(\theta)-0.07\sin(3\theta)+0.2\sin(7\theta)\Big)$$ with $\theta \in [0,2\pi]$. 
An $(N+1)\times (N+1)$ uniform mesh partitioning $[-2,2]\times[-1,3]$ is used. 

\begin{figure}[htb]
\centering
\includegraphics[width=0.32\textwidth,trim={0.2cm 0.2cm 1.9cm 0.2cm},clip]{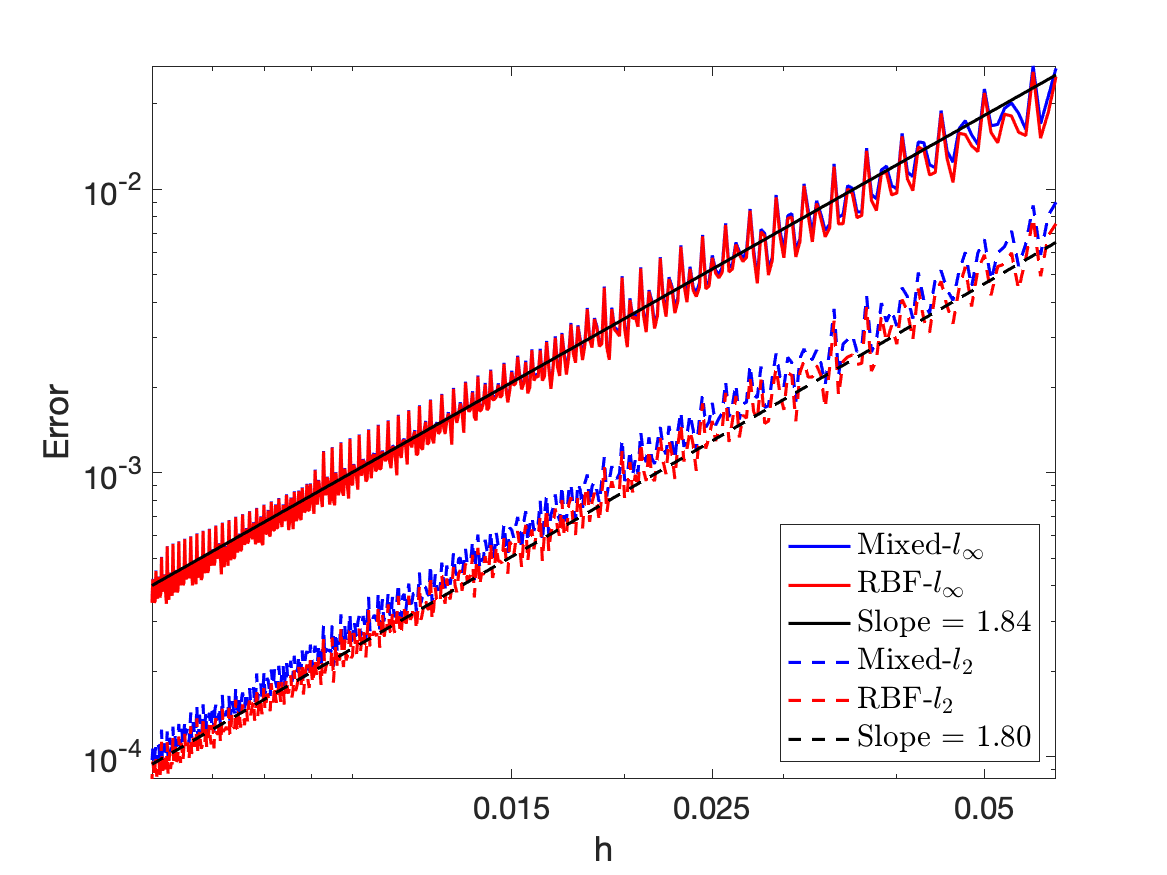}
\includegraphics[width=0.32\textwidth,trim={0.2cm 0.2cm 1.9cm 0.2cm},clip]{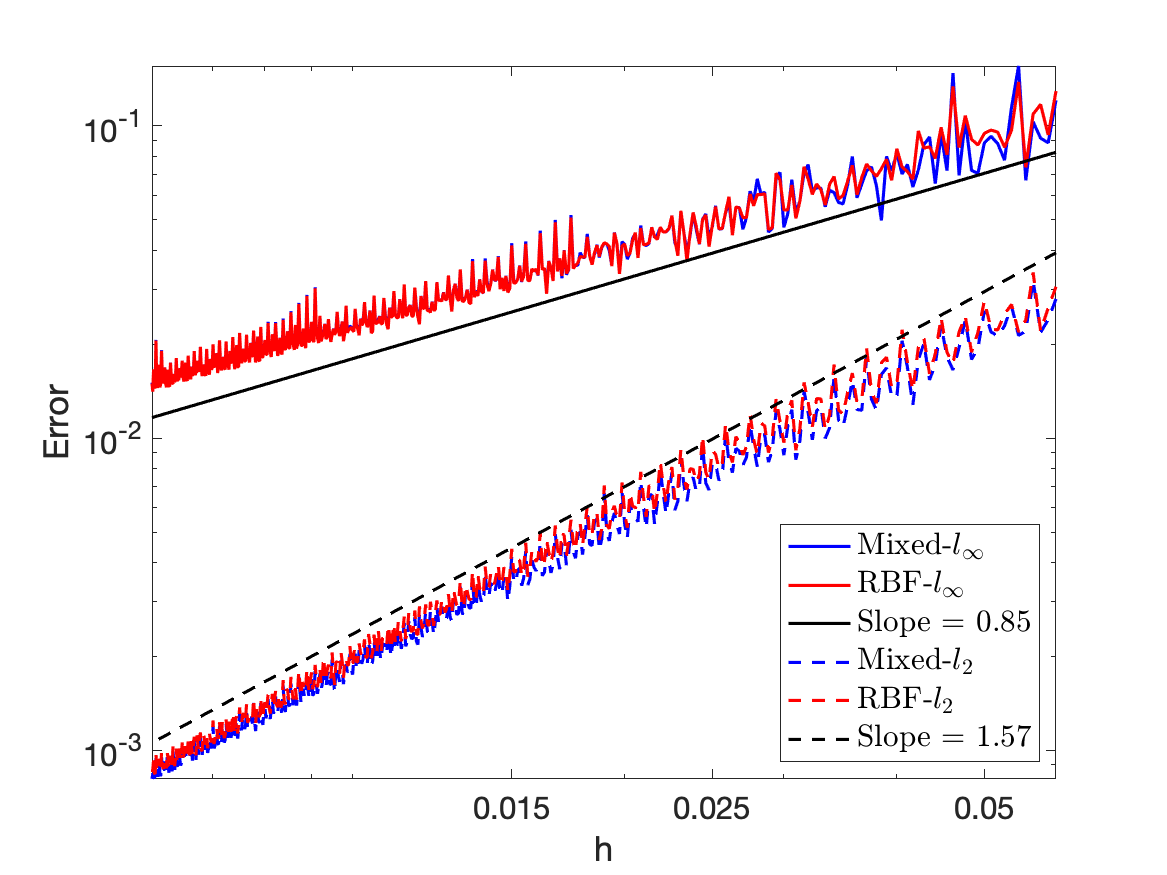}
\includegraphics[width=0.32\textwidth,trim={0.2cm 0.2cm 1.9cm 0.2cm},clip]{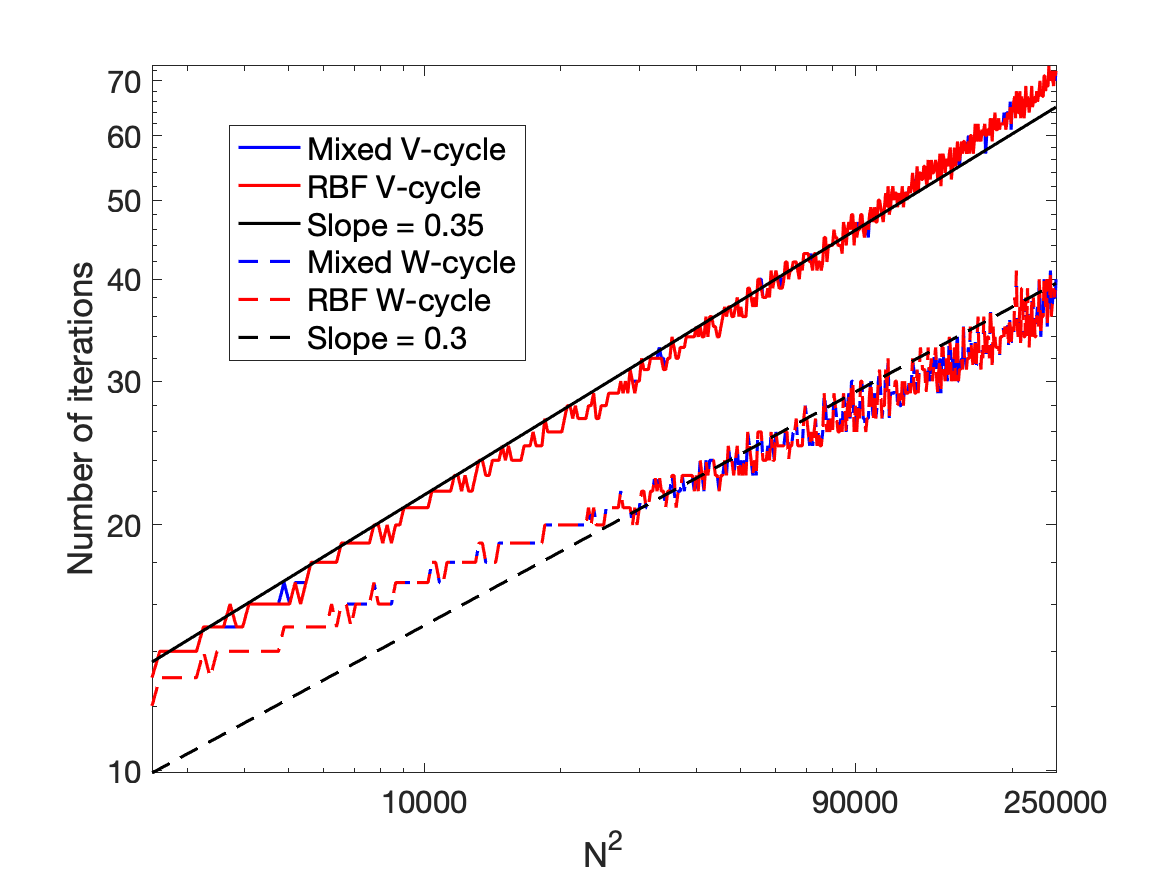}
\caption{Convergence check for the AMG preconditioned embedded boundary method for a bone-shaped domain. Left: Numerical errors as a function of $h$. Middle: Numerical errors of the gradient as a function of $h$. Right: The total number of iterations for convergence for different $N^2$.\label{fig:bone}}
\end{figure}

The numerical solution for $N=1280$ and its difference from the exact solution are presented in \fig{\ref{fig:bone_sol}}. Sweeping from $N=50$ to $N=500$, numerical results are presented in \fig{\ref{fig:bone}}. For both  the mixed and the RBF-based EB methods, we observe $O(h^{1.84})$ error in $l_\infty$ and $O(h^{1.80})$ error in $l_2$ for $u$. The $l_\infty$ error of the gradient $\nabla u$ scale as $O(h^{0.85})$ and the $l_2$ error scale as $O(h^{1.57})$. The total number of iterations needed by the $W$-cycle is smaller than the $V$-cycle. The former scales as $O(DOF^{0.3})$ and  the later scales as $O(DOF^{0.35})$. Except for $4$ out of the $451$ considered discretization sizes, the resulting linear system is diagonally dominant, and the proposed criteria to check the SPD structure is never violated. 
The eigenvalues with the smallest and largest magnitude $\lambda_{\textrm{Small}}$ and $\lambda_{\textrm{Large}}$ are plotted in  \fig{\ref{fig:bone_eig}}. They are all negative, and $\lambda_{\textrm{Small}}$ scales approximately as $O(h^{-2})$. Moreover, $|\lambda_{\textrm{Large}}|$ is smaller than $24$, which is the eigenvalue with largest magnitude of the problem with periodic boundary condition and conductivity coefficient $\widetilde{\beta}=\max(|\beta(x,y|)=3$. 

\begin{figure}[htb]
\centering
\includegraphics[width=0.45\textwidth,trim={0.2cm 0.2cm 1.0cm 0.2cm},clip]{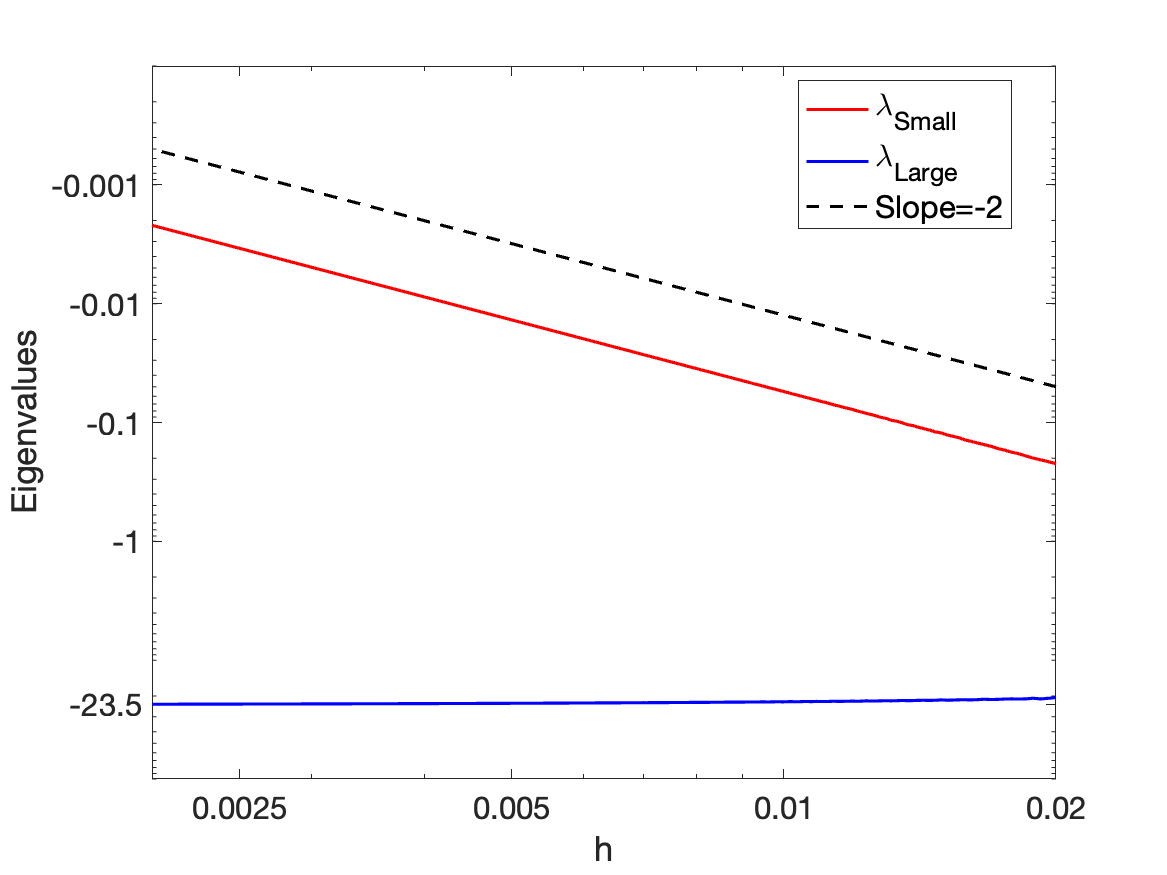}
\includegraphics[width=0.45\textwidth,trim={0.2cm 0.2cm 1.0cm 0.2cm},clip]{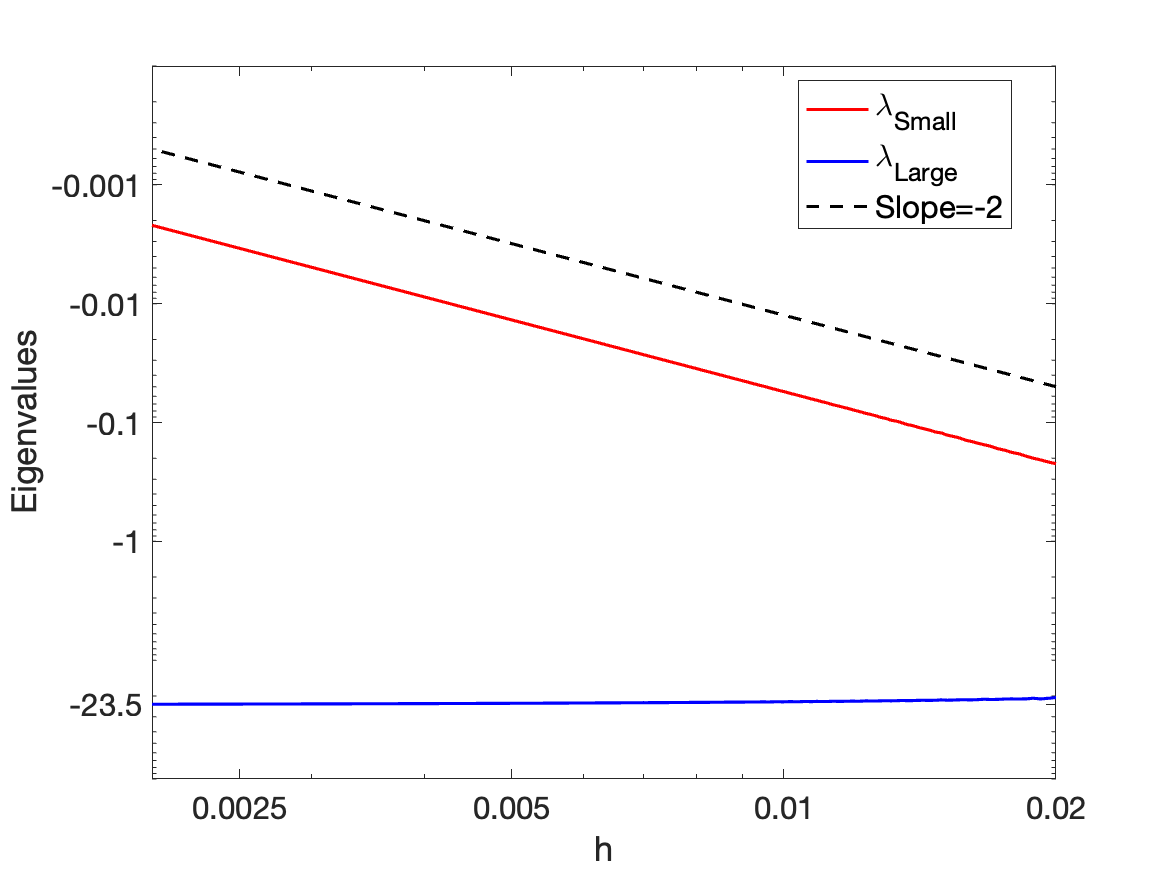}
\caption{Bone-shaped geometry: the maximum and minimum eigenvalues of the discrete Laplace operator as a function of $h$. Left: the embedded boundary method using mixed interpolation. Right: the embedded boundary method using RBF-based interpolation.\label{fig:bone_eig}}
\end{figure}

\subsubsection{The Helmholtz equation in a star-shaped geometry}\label{sec:starshape}
We consider the star-shaped geometry described in \cite{gibou2002second}. We partition the computational domain $[-1,1] \times [-1,1]$ by a $(N+1)\times(N+1)$ uniform mesh. The  boundary interface is parameterized by 
$$\Big(x(\theta),y(\theta)\Big)=\Big(0.02\sqrt5+(0.5+0.2\sin(5\theta))\cos(\theta),0.02\sqrt5+(0.5+0.2\sin(5\theta))\sin(\theta)\Big)$$ with $\theta \in [0,2\pi]$, and the center of this star shape is $(0.02\sqrt{5},0.02\sqrt{5})$. We solve the Helmholtz equation 
\begin{align}
\omega^2 u + \Delta u = 1000\delta(x_s)\delta(y_s),\quad \omega = 50,
\end{align}
with zero Dirichlet boundary conditions. Here, $(x_s,y_s)=(-0.375,0.125)$ and $\delta$ is the Dirac delta function which is approximated by a two dimensional hat function whose integral is one. The Helmholtz operator $\omega^2I+\Delta$ is indefinite,  and we use the MINRES method as our linear solver. The numerical solution for $N=800$ is presented in \fig{\ref{fig:star_eig}}.

\begin{figure}[htb]
\centering
\includegraphics[width=0.7\textwidth,trim={0.2cm 0.2cm 1.7cm 0.2cm},clip]{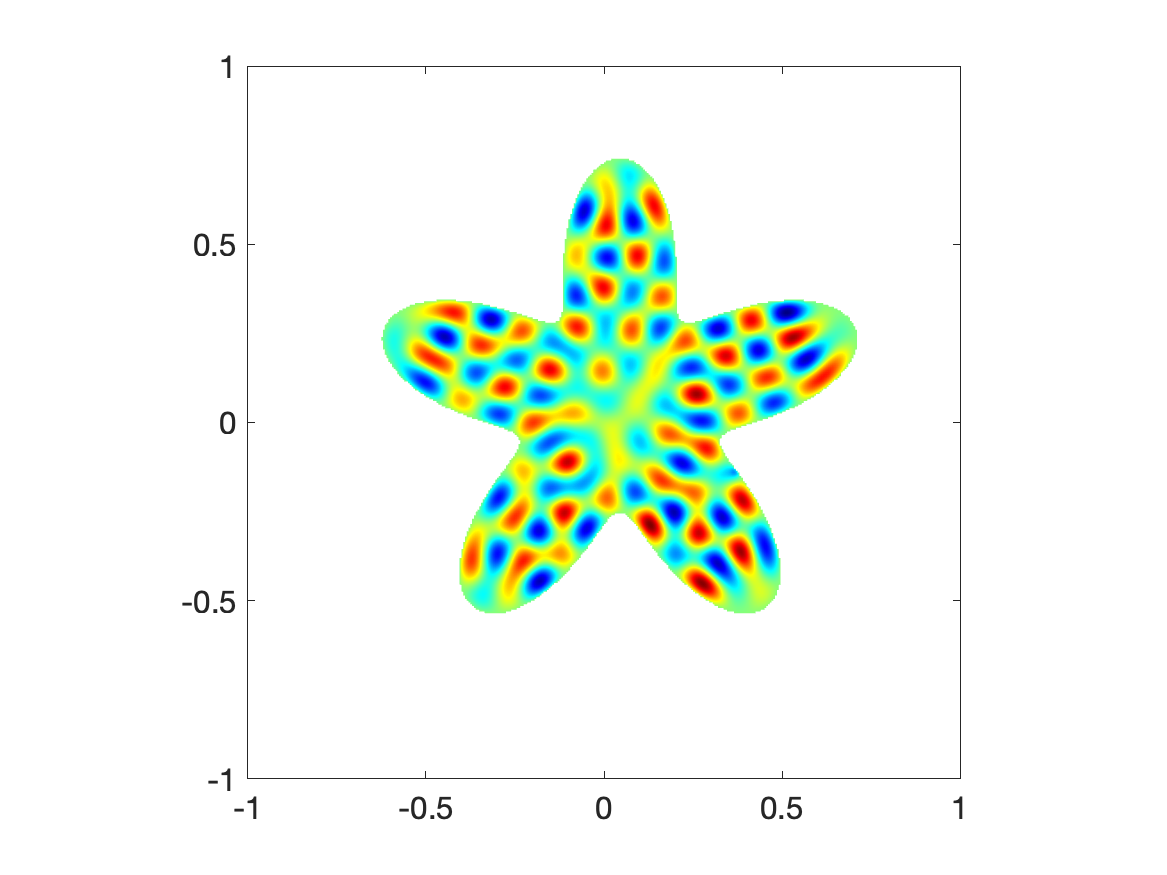}
\caption{The Helmholtz equation in two dimensions, $\omega^2 u + \Delta u =f$, with Dirichlet boundary conditions in a star-shaped domain. Plotted is the numerical solution with $N=800$.\label{fig:star_eig}}
\end{figure}


\subsection{Heat equation}
In this example, we consider the heat equation with Dirichlet boundary conditions
\begin{subequations}
\begin{align}
&u_t = \nabla\cdot(\beta \nabla u)+f,\quad (x,y)\in\Omega,\\
 &u(x,y,t) = u_{\mathcal{D}}(x,y),\quad (x,y)\in \Gamma \quad\text{and}\quad u(x,y,0) = u_{\mathcal{I}}(x,y).
\end{align}
\end{subequations}
We consider $\beta(x,y) = 0.25-x^2-y^2$, and impose a homogenous Dirichlet boundary condition. The geometry $\Omega$ is a disk determined by the level set function $\psi(x,y)=0.25-x^2-y^2$. The source $f(x,y,t)$ and the initial condition are chosen such that $u(x,y,t)=e^{-t}(x^2+y^2-0.25)$. The computational domain is $[-1,1]\times[-1,1]$ partitioned by a $(N+1)\times(N+1)$ uniform mesh.
 
We use our method as the spatial discretization and the Crank-Nicolson method as the time integrator with a time step size $\Delta t=h$. In this example only the $V$-cycle is used in the AMG. When inverting the linear system, we use the data from last step as initial guess. We solve the problem for time $t\in[0,0.5]$. The numerical solution with $N=1000$ at $t=0.5$ is presented in \fig{\ref{fig:heat}}. The $l_2$ and $l_\infty$ errors in $u$ at $t=0.5$ are converging as $O(h^2)$, the error of the gradient $\nabla u$ at $t=0.5$ converges as $O(h)$ in $l_\infty$ and  as $O(h^{1.45})$ in $l_2$. On average, we need $2$ iterations for convergence per time step.
\begin{figure}[htb]
\centering
\includegraphics[width=0.32\textwidth,trim={0.2cm 0.2cm 1.7cm 0.2cm},clip]{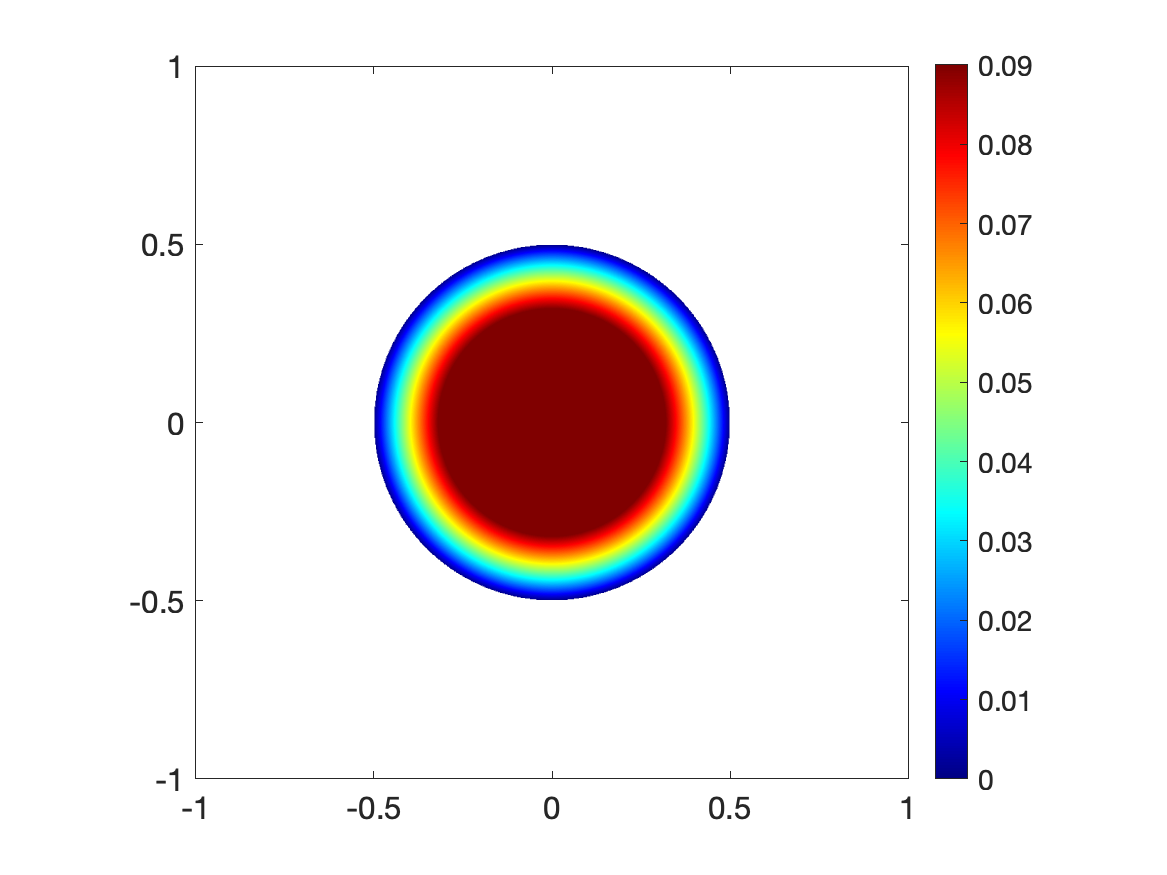}
\includegraphics[width=0.32\textwidth,trim={0.2cm 0.2cm 1.45cm 0.2cm},clip]{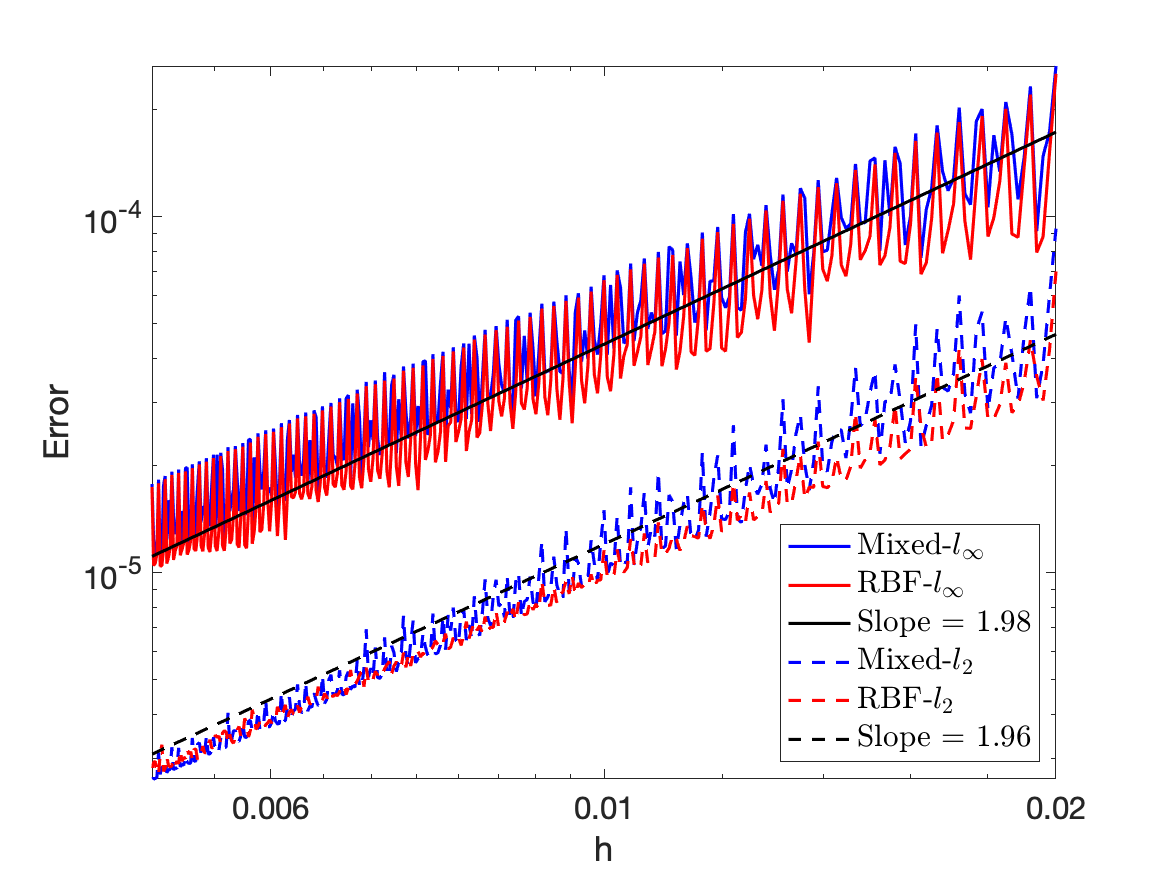}
\includegraphics[width=0.32\textwidth,trim={0.2cm 0.2cm 1.45cm 0.2cm},clip]{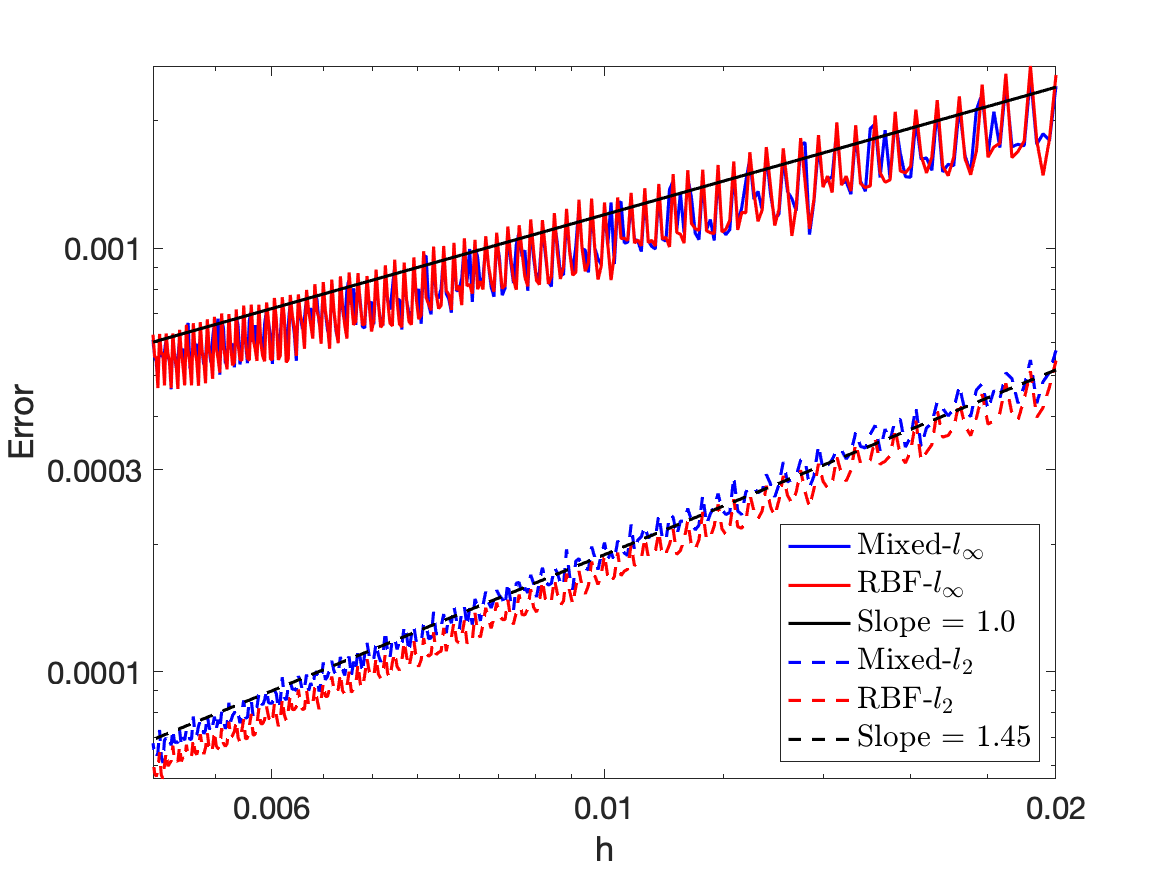}
\caption{The heat equation in two dimensions, $u_t = \nabla\cdot(\beta \nabla u)+f$, with Dirichlet boundary conditions on a disk-shaped domain. Left: Numerical solution using $N=1000$. Middle: Numerical error as a function of $h$ at the final time $T=0.5$. Right: Numerical errors of the gradient as a function of $h$ at the final time $T=0.5$. \label{fig:heat}}
\end{figure}


\subsection{The wave equation}
In this example we solve the wave equation with homogenous Dirichlet boundary conditions 
\begin{eqnarray*}
&\frac{\pa^2 u(x,y,t)}{\pa t^2} = \nabla^2 u(x,y,t), &  (x,y) \in \Omega \equiv \{ x^2+y^2<1 \}, \ \ t>0, \\ 
&u(x,y,t) = 0, &  (x,y) \in \Gamma \equiv \{ x^2+y^2=1 \}.  
\end{eqnarray*}
The initial data is chosen so that the solution is a standing mode
\begin{equation}\label{eq:Bessel}
u_{mn}(r,\theta,t) = J_m(r \kappa_{mn}) \cos m \theta \, \cos \kappa_{mn} t.
\end{equation}
Here $J_m(z)$ is the Bessel function of the first kind of order $m, (m=0,1,\ldots)$ and $\kappa_{mn}$ is the $n$th zero of $J_m$. In this problem we set $m=n=7$. Then $\kappa_{77}=31.4227941922$ and the period of the solution is $\f{2\pi}{\kappa_{77}} = 0.1999562886971$.

To discretize in time we use the so-called $\theta$-scheme (see e.g. \cite{britt2018high})
\begin{equation}
u^{n+1} - \theta \Delta t^2 \nabla^2 u^{n+1} = 2u^n + (1-2\theta) \Delta t^2 u^n - u^{n-1} + \theta \Delta t^2 \nabla^2 u^{n-1}. 
\end{equation}
Here $\theta \ge 0$ is a parameter that can be chosen to obtain different schemes. The values we consider here are $\theta=0$, which corresponds to the classic explicit leap-frog scheme, $\theta = 1/2, 1/4$ which corresponds to second order unconditionally stable implicit methods, and $\theta = 1/12$ which leads to an implicit method that is fourth order accurate but with a stability constraint on the time step. Comparing the explicit method and the fourth order accurate method, the latter can march with a time step that is roughly 50\% larger than the second order explicit method.
We discretize the domain $(x,y) \in [-1.1,1.1]^2$ using 100, 200, 400 and 800 grid points and evolve the numerical solution for 10.2 periods. The mixed EB method is used in the spatial discretization. We use $\Delta t /h = (0.7, 2 ,2 ,0.85)$ for the schemes corresponding to $\theta = (0,1/2,1/4,1/12)$. The initial data are set using the exact solution. In Table \ref{tab:wave} we report max-errors, along with computed rates of convergence (using two subsequent refinement levels) and average number of CG-AMG iterations per time step, at the final time. It is clear that for this example there is no benefit to use an unconditionally stable time stepping method as the temporal errors are dominating the total error if $\Delta t > h$. On the other hand, the fourth order method achieves an error that is roughly equivalent with the explicit scheme and may be a suitable alternative for problems with solutions that vary rapidly in time. 

In Figure \ref{fig:wavesol} and Figure \ref{fig:waveerr} we display the numerical solution and the error for the four different methods.

\begin{table}[htb]
\begin{center}
\begin{tabular}{|c|c|c|c|c|c|c|c|c|c|c|c|}
\hline
$h$  
&  $\theta=0$ 
& $p$ 
& $\theta=\frac{1}{2}$ & $p$ & iter 
& $\theta=\frac{1}{4}$ & $p$ & iter   
& $\theta=\frac{1}{12}$ & $p$ & iter \\
\hline
\hline
 2.2(-2) & 3.6(-2) & $\ast$ & 2.0(-1) & $\ast$ & 6.0 &5.7(-1) & $\ast$ & 6.0 &5.9(-2) & $\ast$ & 3.0 \\ 
1.1(-2) & 9.2(-3) & 2.0 & 4.8(-1) & -1.3 & 6.0 & 7.6(-3) & 6.2 & 5.2 & 2.1(-2) & 1.5 & 3.0 \\
5.5(-3) & 2.5(-3) & 1.9 & 6.0(-2) & 3.0 & 6.0 & 5.1(-2) & -2.7 & 5.0 & 6.1(-3) & 1.8 & 3.0\\
2.75(-3) & 6.7(-4) & 1.9 & 3.5(-2) & 0.8 & 6.0 & 1.6(-2) & 1.7 & 5.0 & 1.6(-3) & 1.9 & 3.0\\
1.375(-3) & 1.7(-4) & 2.0 & 9.9(-3) & 1.8 & 5.0 & 4.3(-3) & 1.9 & 5.0 & 4.1(-4) & 2.0 & 3.0\\
\hline
\end{tabular}
\end{center}
\caption{The table displays max-errors, along with computed rates of convergence (using two subsequent refinement levels) and average number of CG-AMG iterations per time step, at the final time $T$. The results are for the  wave equation example \label{tab:wave}}
\end{table}%

\begin{figure}[htb]
\begin{center}
\includegraphics[width=0.49\textwidth,trim={0.5cm 0.3cm 0.4cm 0.6cm},clip]{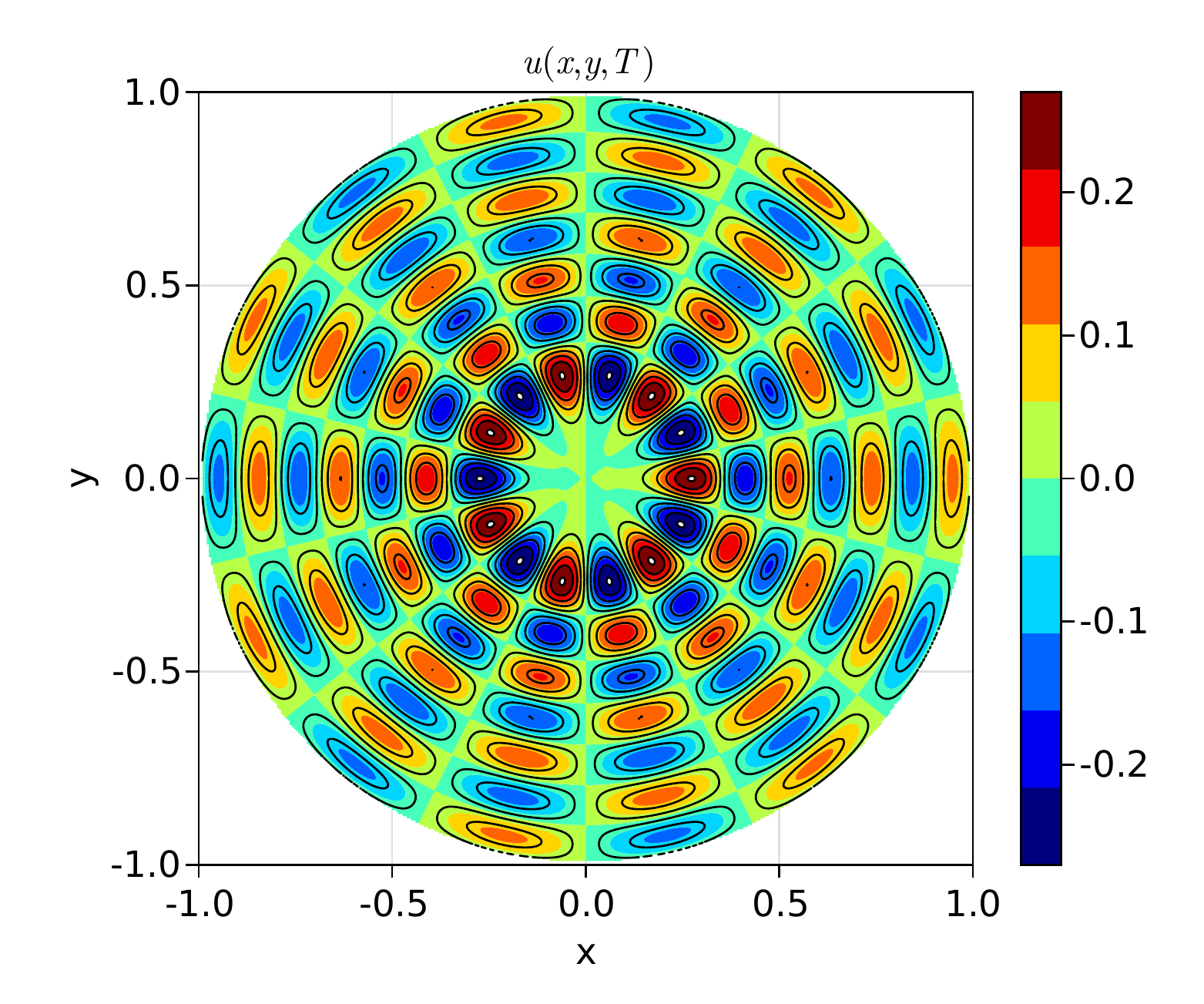}
\includegraphics[width=0.49\textwidth,trim={0.5cm 0.3cm 0.4cm 0.6cm},clip]{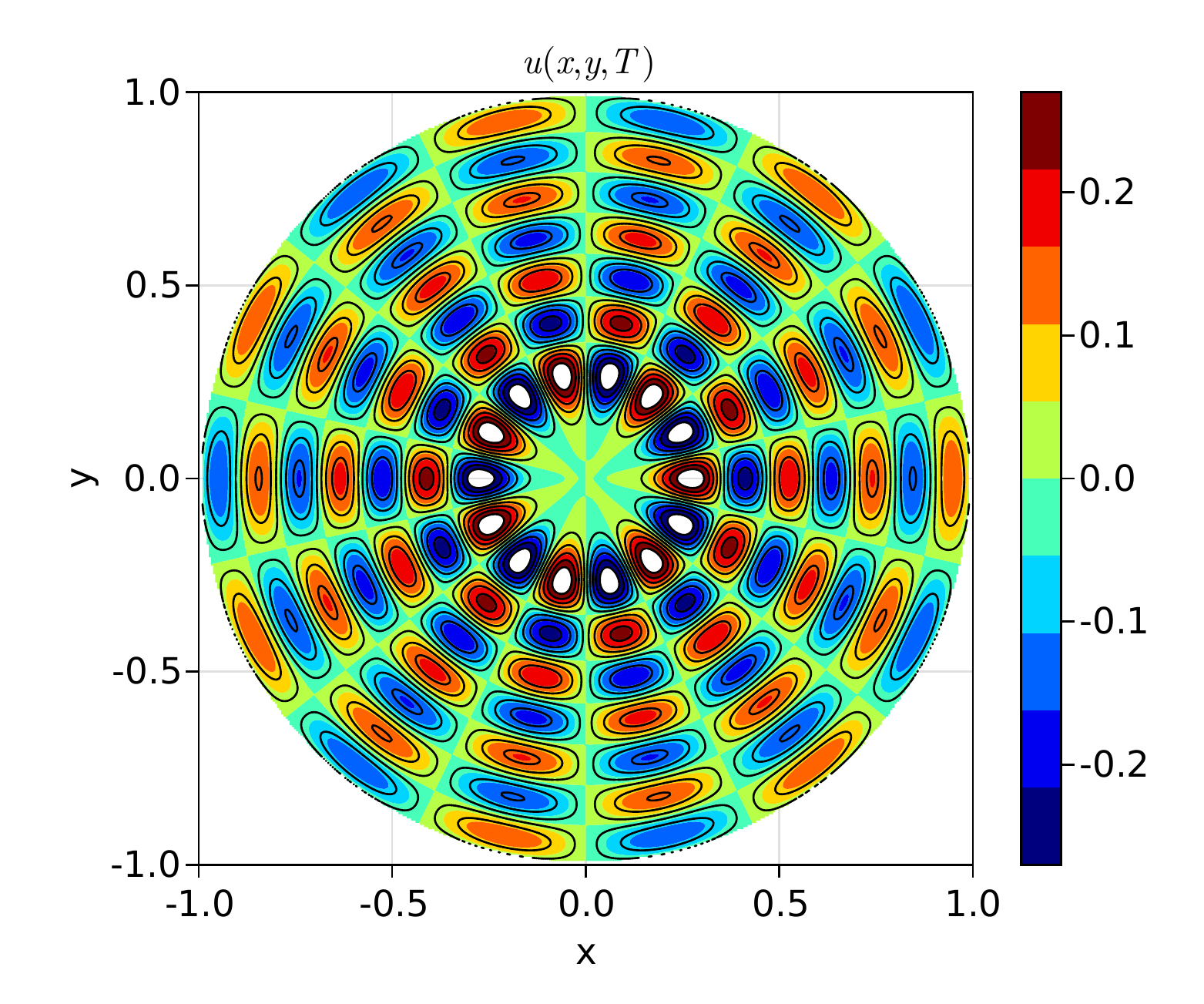}
\includegraphics[width=0.49\textwidth,trim={0.5cm 0.3cm 0.4cm 0.6cm},clip]{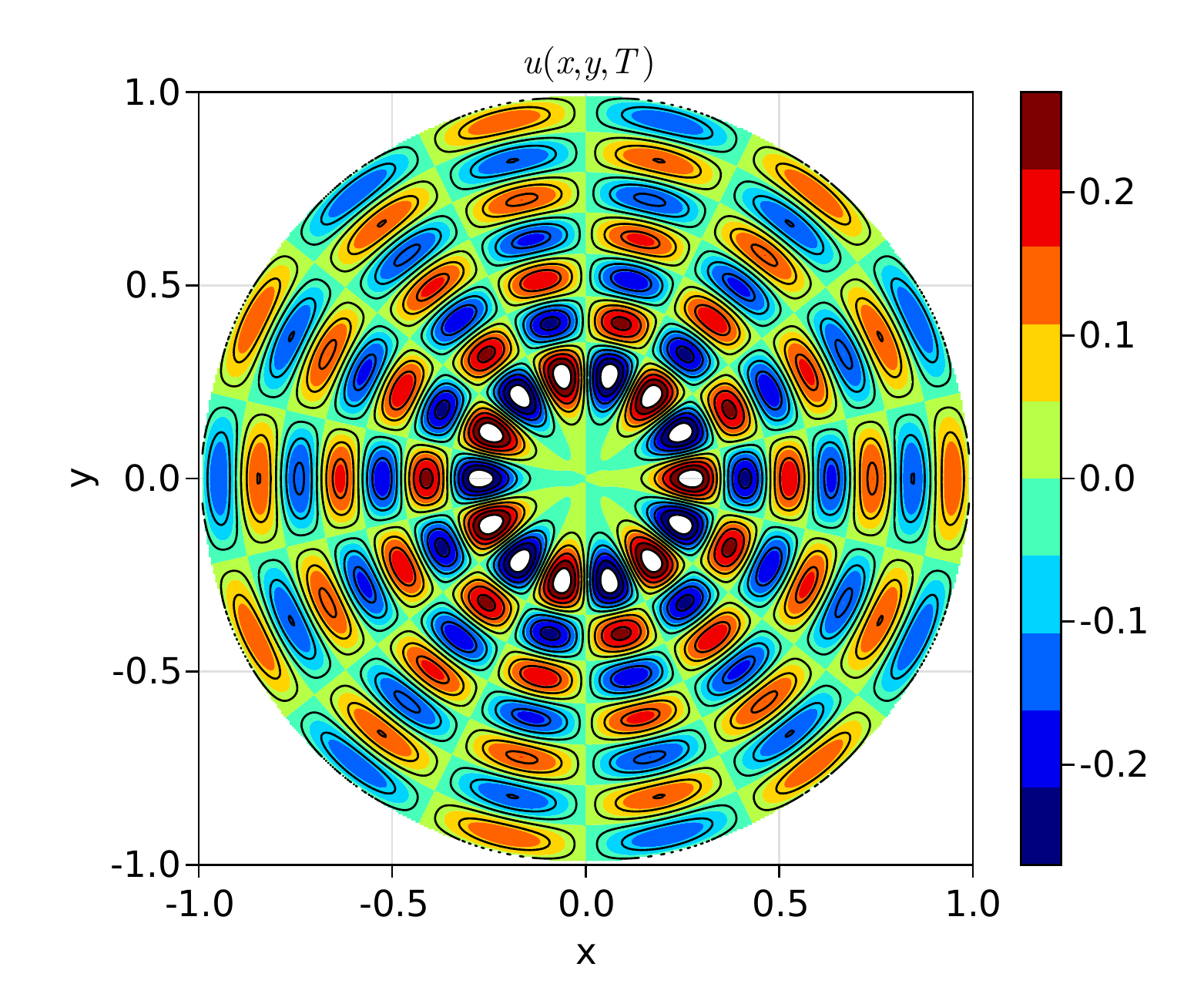}
\includegraphics[width=0.49\textwidth,trim={0.5cm 0.3cm 0.4cm 0.6cm},clip]{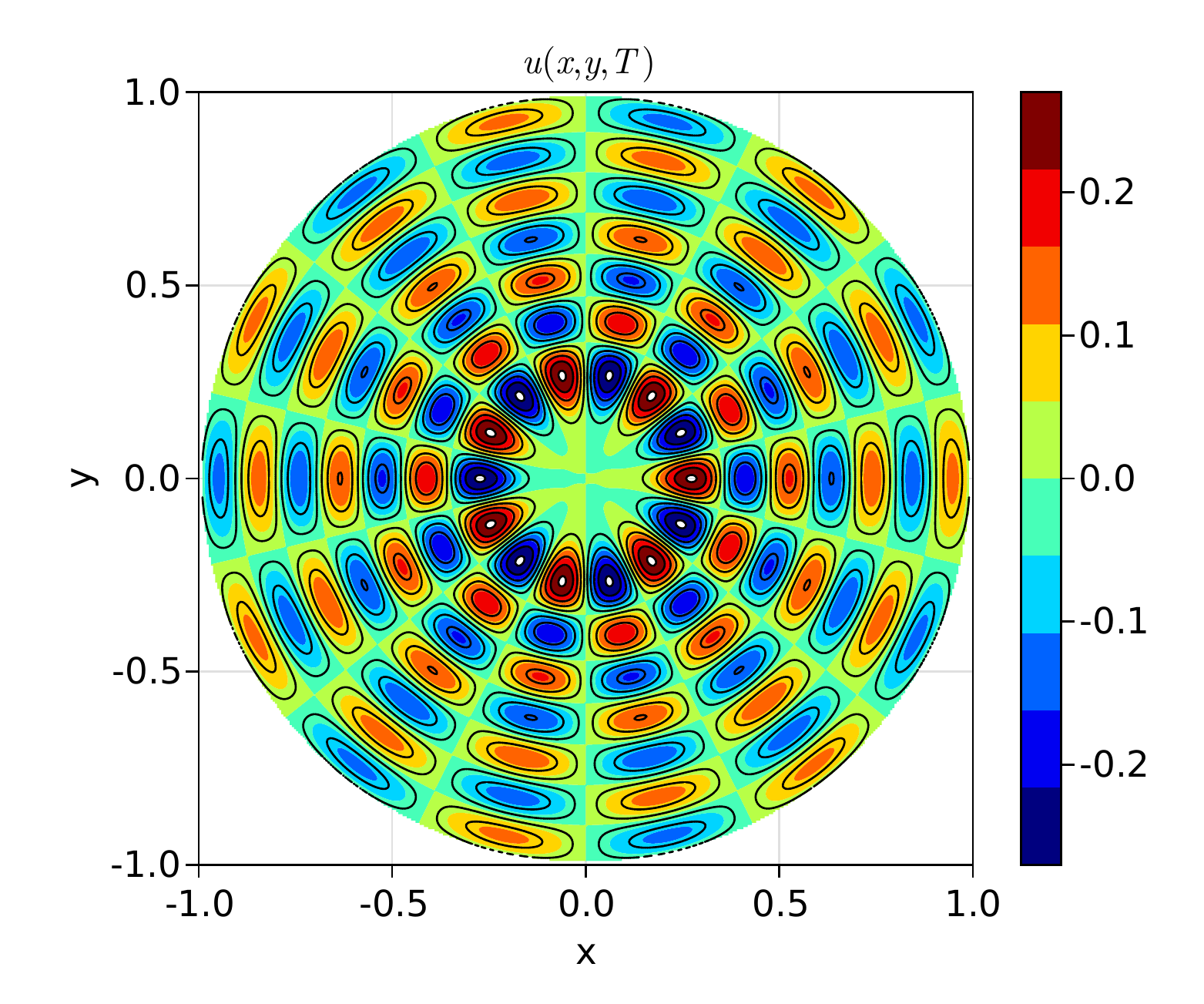}
\caption{The wave equation, $u_{tt} = \nabla^2 u$, with zero boundary condition on a disk shaped geometry. Numerical solutions with $\theta = (0,1/2,1/4,1/12)$ correspondingly from left to right, from top to bottom at the final time.\label{fig:wavesol}}
\end{center}
\end{figure}

\begin{figure}[htb]
\begin{center}
\includegraphics[width=0.49\textwidth,trim={0.5cm 0.3cm 0.4cm 0.6cm},clip]{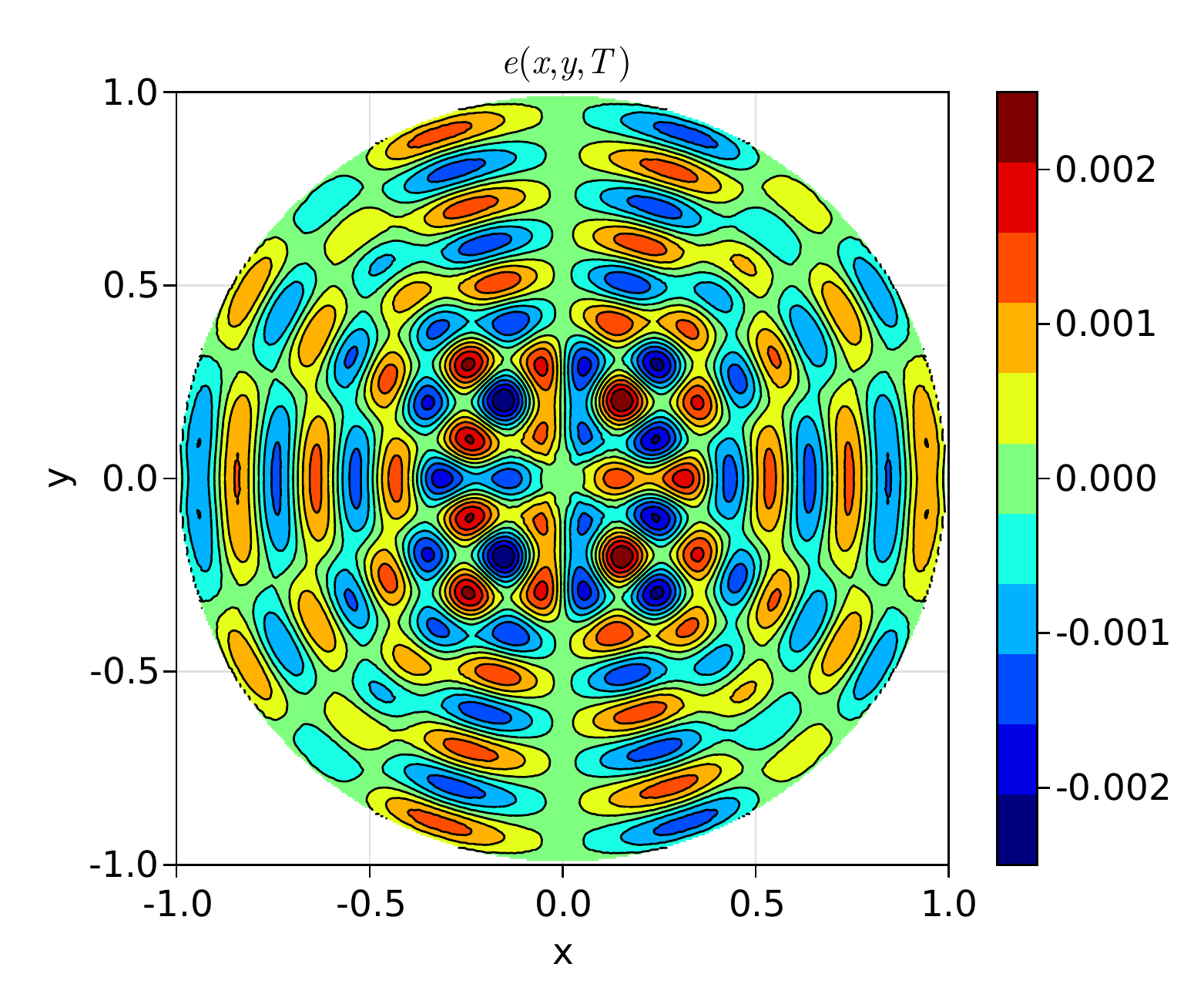}
\includegraphics[width=0.49\textwidth,trim={0.5cm 0.3cm 0.4cm 0.6cm},clip]{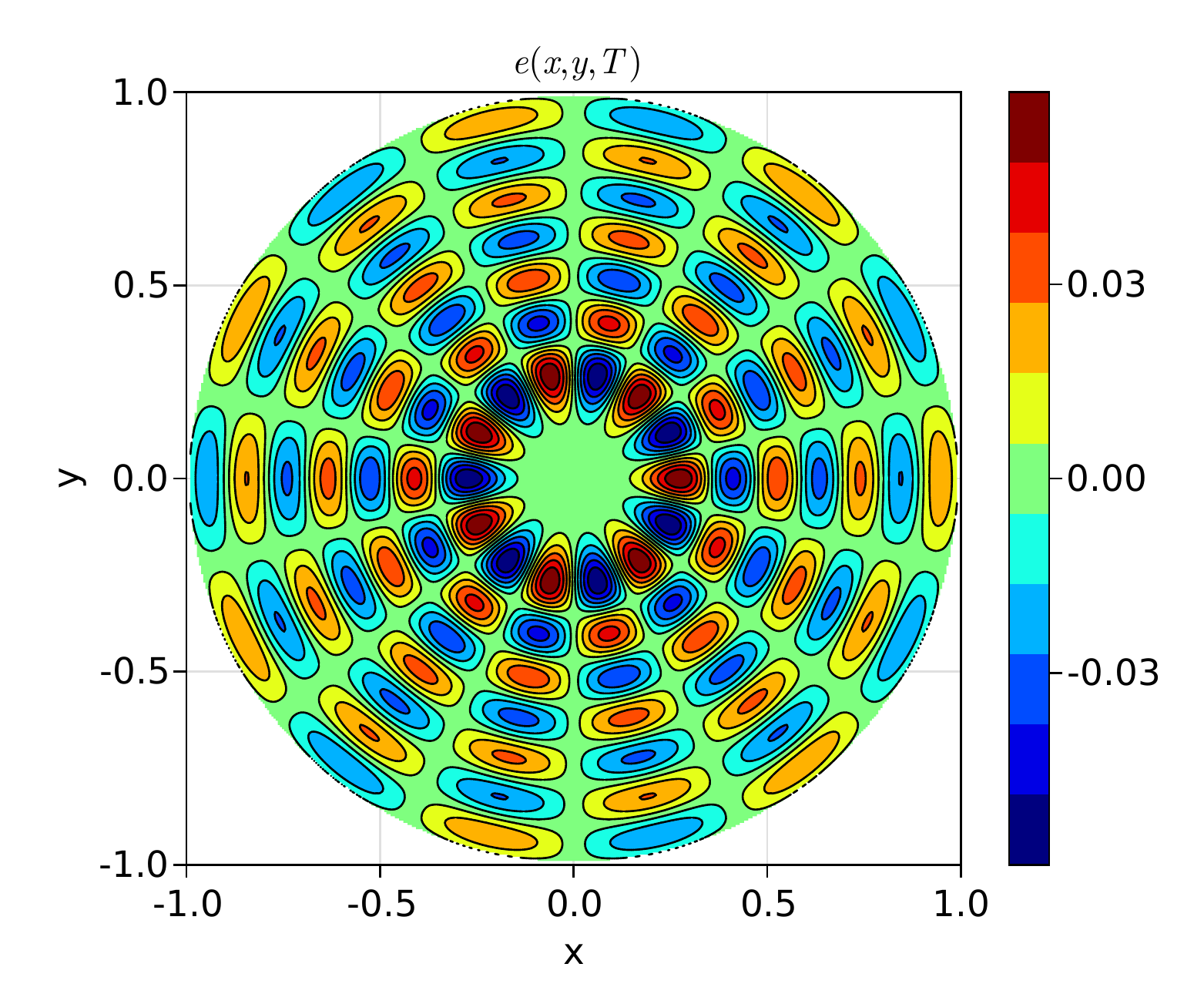}
\includegraphics[width=0.49\textwidth,trim={0.5cm 0.3cm 0.4cm 0.6cm},clip]{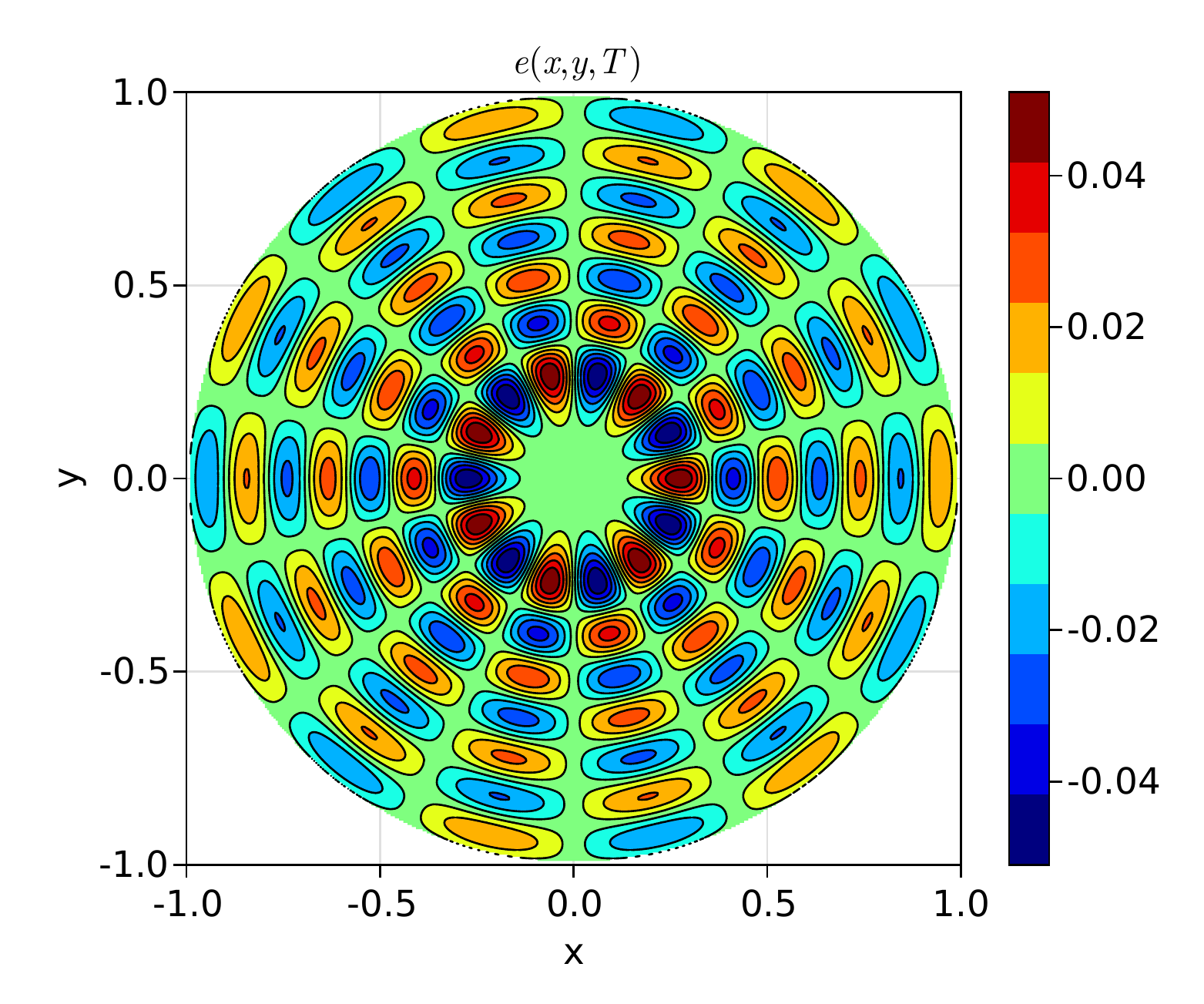}
\includegraphics[width=0.49\textwidth,trim={0.5cm 0.3cm 0.4cm 0.6cm},clip]{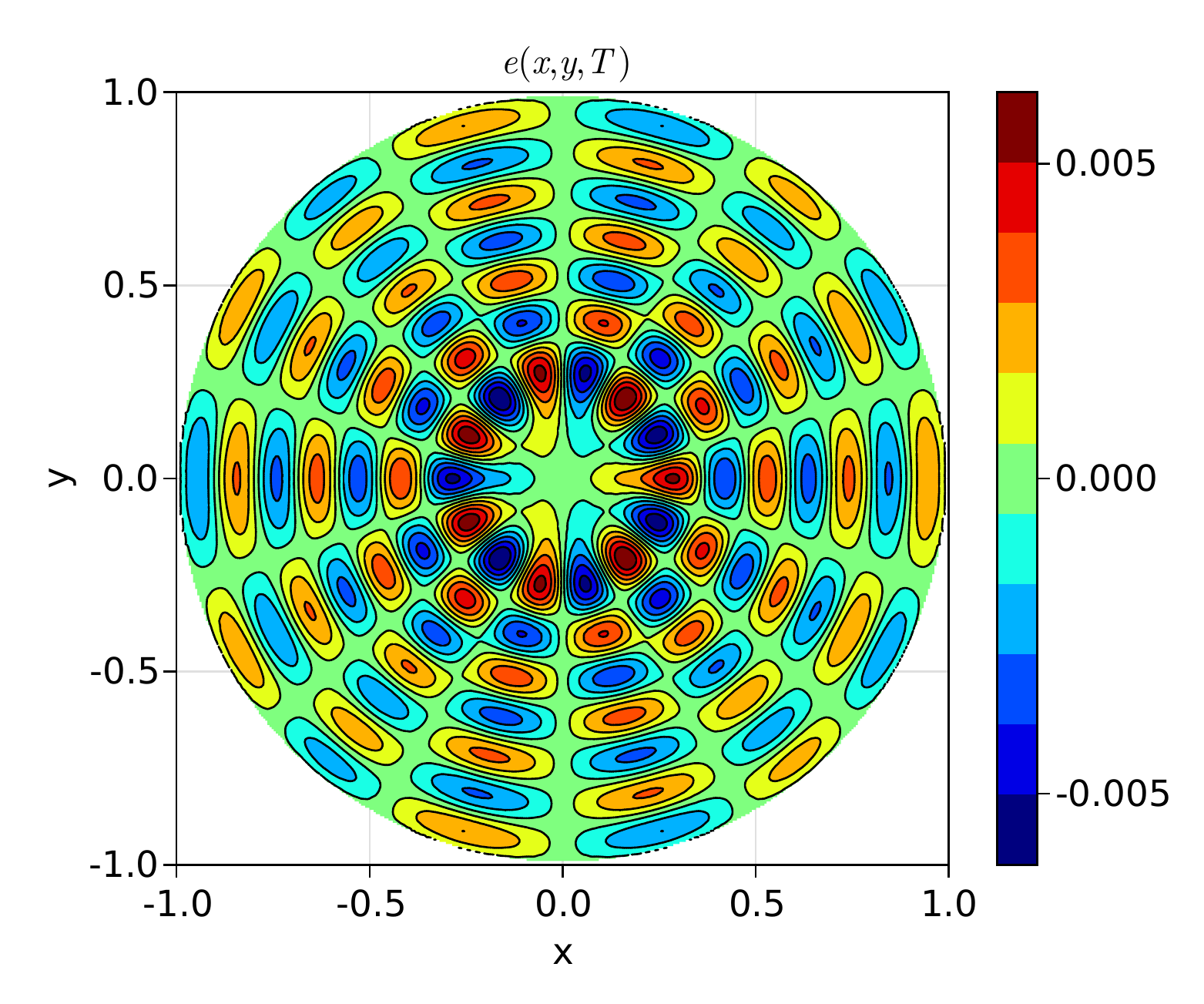}
\caption{The wave equation, $u_{tt} = \nabla^2 u$, with zero boundary condition on a disk shaped geometry. Displayed are the errors for $\theta = (0,1/2,1/4,1/12)$ correspondingly from left to right, from top to bottom at the final time. \label{fig:waveerr}}
\end{center}
\end{figure}


\subsection{Quantity of interest determined by geometry parameters}
In shape optimization  problems, a quantity of interest (QOI) or loss may be determined by the parameters of the geometry. In order to find the optimal geometry, the gradient of the QOI or loss with respect to the parameters needs to be computed. If one wants to use an embedded boundary method as a building block for shape optimization the QOI computed by the method should be smooth with respect to the parameter. To show the potential of our method for these type of problems, we consider two prototype examples.

In the first example, we consider Poisson's equation on an elliptical shape with a fixed area
\begin{equation*}
	\frac{x^2}{a^2} + \frac{y^2}{1/a^2} =1.
\end{equation*}
Here, $a\in[0.5,2]$, $\beta=1$ and zero Dirichlet boundary conditions are used. The source function is $f(x,y)=3$. The QOI is the value of $u$ at point $(0,0)$. In \fig{\ref{fig:QOI}}, we present the  QOI and compute its derivative with respect to $a$ by the central difference method. Both the QOI and computed derivative are smooth. We observe that $u(0,0)$ has the minimum at $a=1$, when the elliptical shape becomes a unit circle. We also observe that the value of $u(0,0)$ corresponding to $a$ and $\frac{1}{a}$ equal to each other.

In the second example, we consider an ellipsoid that is rotated by an angle $\alpha$ and the QOI is the integration of $u$ in the rectangular domain $[-1,1]\times[-0.5,0.5]$. The rotated-ellipse-shape  geometry is determined by the level-set function
\begin{align}
	\psi(x,y) = \frac{(x \cos(\theta)-y \sin(\theta))^2}{16} + \frac{(x \sin(\theta)+y \cos(\theta))^2}{4} -1.
\end{align}
where $\theta\in[0,2\pi)$. We consider the Poisson equation with $\beta=1$ and zero Dirichlet boundary conditions, and the source function is $f(x,y)=3$.
 A $501\times 501$ uniform mesh partitioning $[-5,5]\times[-5,5]$ is used.
In \fig{\ref{fig:QOI}}, we present the  QOI and compute its derivative with respect to $\theta$ by the central difference method. Both of the QOI and the derivative are smooth. Expected symmetry is also observed.

For both examples, the QOI and its derivative computed by our method are smooth, and the underlying geometric symmetry is also respected.
\begin{figure}[htb]
\centering
\includegraphics[width=0.45\textwidth,trim={0.2cm 0.2cm 1.0cm 0.2cm},clip]{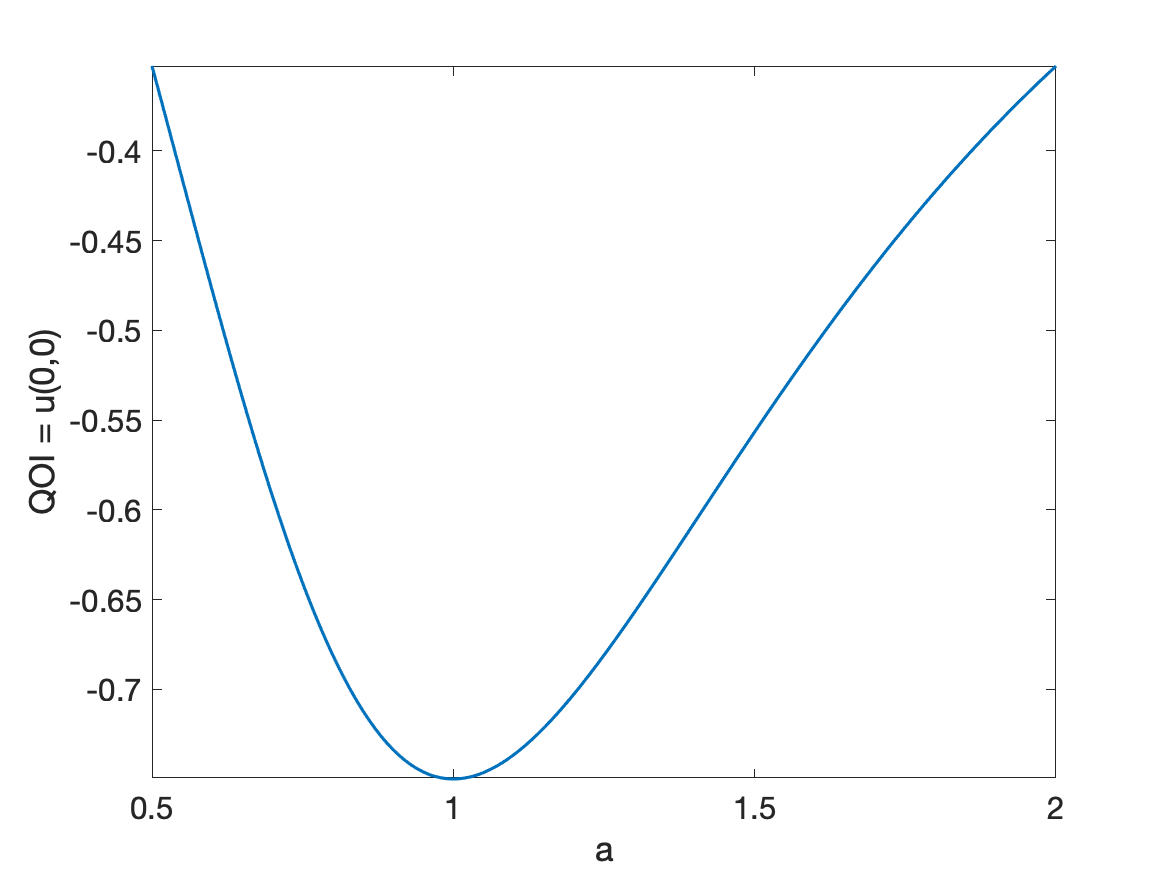}
\includegraphics[width=0.45\textwidth,trim={0.2cm 0.2cm 1.0cm 0.2cm},clip]{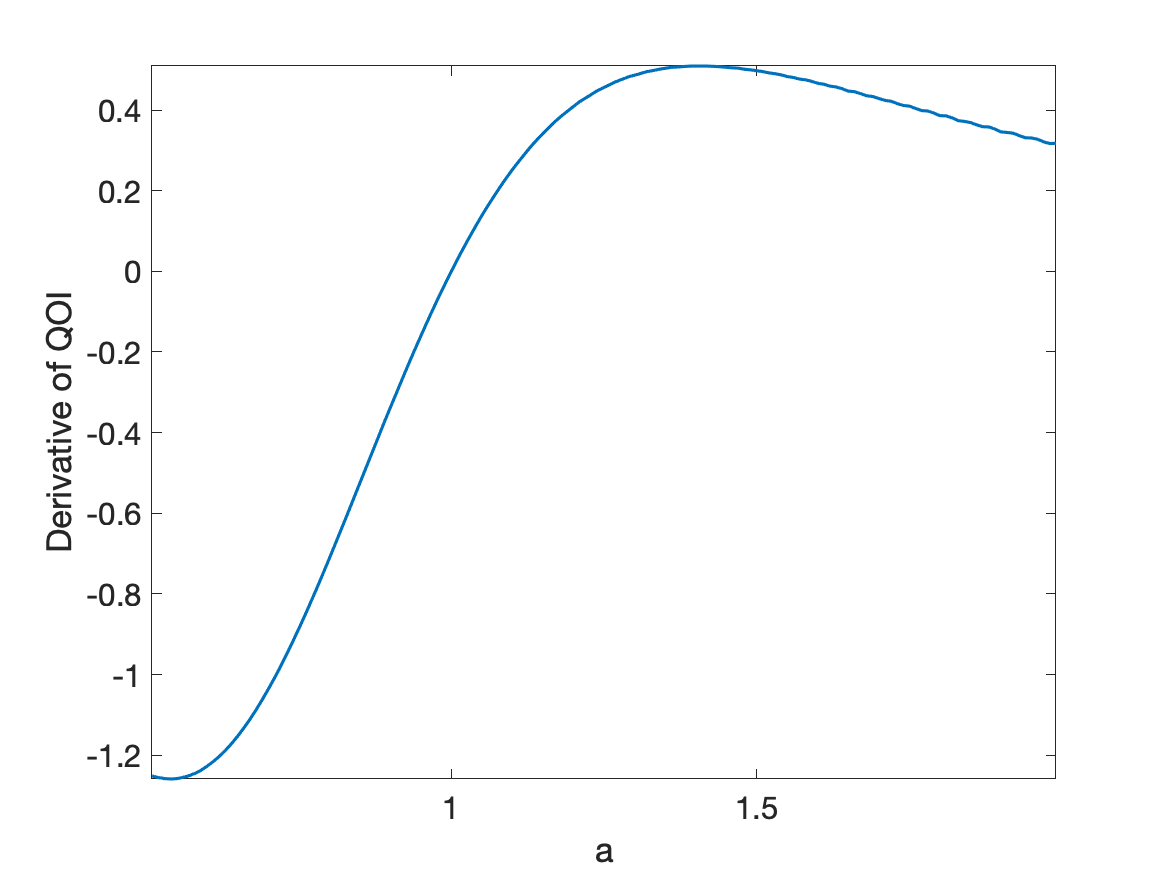}
\includegraphics[width=0.45\textwidth,trim={0.2cm 0.2cm 1.0cm 0.2cm},clip]{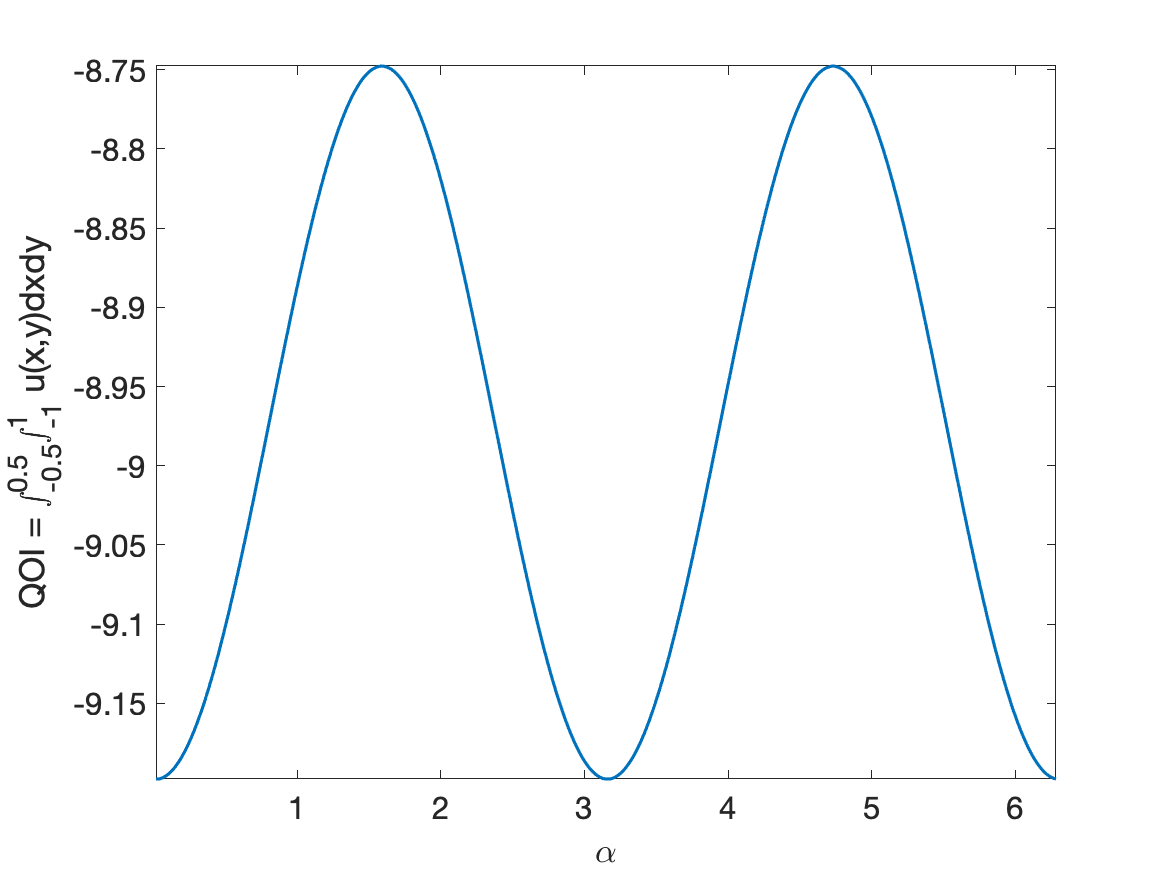}
\includegraphics[width=0.45\textwidth,trim={0.2cm 0.2cm 1.0cm 0.2cm},clip]{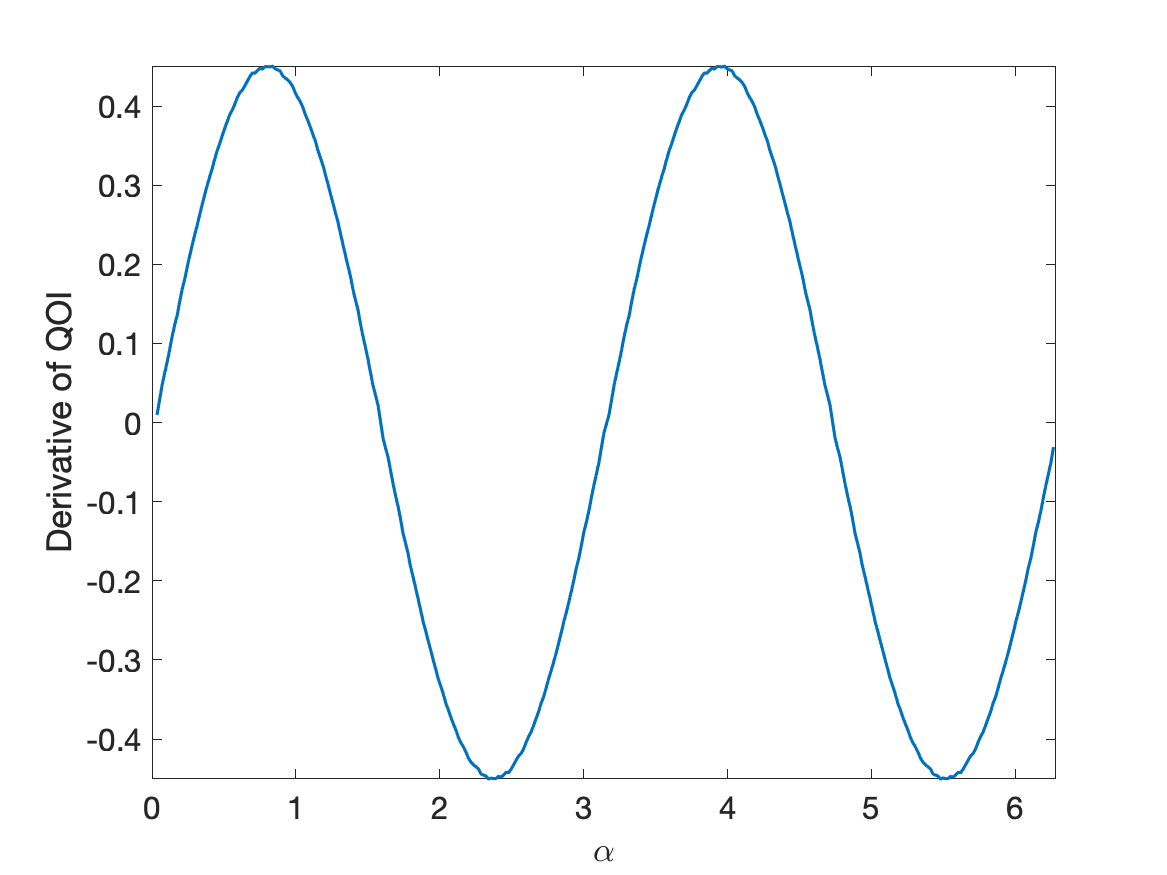}
\caption{Computations of QOIs with geometry parameter. Top left: QOI vs $a$ for the first example. Top right: the derivative of QOI with respect to $a$ vs $a$ for the first example. Bottom left: QOI vs  the rotation angle $\alpha$ for the second example. Bottom right: the derivative of QOI with respect to $\alpha$ vs  the rotation angle $\alpha$ for the second example. \label{fig:QOI}}
\end{figure}

\section{Conclusion\label{sec:conclusion}}
In summary, we have developed a universal embedded boundary method to solve Poisson's equation, the heat equation or the wave equation in general two-dimensional domains with complex geometries subject to Dirichlet boundary conditions. The two advantages of this method are: (1) by using interior boundary points instead of placing the ghost point outside the computational domain, the small-cell stiffness  is avoided. (2) to impose the boundary condition, we  apply the line-by-line interpolation or RBF interpolation, which results in a SPD linear system. The SPD structure of the discrete Laplacian operator can be rigorously proved for the convex geometry and is numerically verified for non-convex geometries. A simple criteria to a-priori check the SPD property is also proposed. 

A possible next step in our work is to generalize the current method to  problems with the Neumann boundary conditions. Other possible extensions include combining the current method with the level set method for moving boundary problems (e.g. Stefan problem) and the generalization of the method to the Navier-Stokes equations, and the Grad-Shafranov equations \cite{liu2021parallel,Peng2020AnAD}. 

\appendix
\section{Proof of Lemma \ref{lem:1d_spd}\label{appendix:proof}}
For the case $n=1$ and $n=2$, the lemma can be verified through direct computations of leading principal minors. We only focus on the case $n\geq 3$.

The matrix $\mathcal{D}^{(n)}(a,b)$ is manifestly symmetric. To prove that $\mathcal{D}^{(n)}\in\mathbb{R}^{n\times n}$ is SPD, we use mathematical induction to prove that its leading principal minors are all positive. When $1\leq k\leq n-1$, the $k$-th order leading principal minor of $\mathcal{D}^{(n)}\in\mathbb{R}^{n\times n}$ is
\begin{align}
Q^{(n)}_k(a) = \det\underbrace{\left(\begin{matrix}
a & -1 & 0 & \dots & 0 & 0\\
-1 &  2 &-1& \dots  & 0 & 0\\
0 & -1 & 2 & \dots  & 0 & 0\\
\vdots & &\vdots & \vdots & \vdots & \vdots\\
0 & 0 & 0 & \dots & 2 & -1\\
0 & 0 & 0 & \dots &-1  & 2
\end{matrix}\right)}_{k\times k}, \qquad 1\leq k\leq n-1.\label{eq:Qan_def}
\end{align}
The $n$-th order leading principal minor of $\mathcal{D}^{(n)}(a,b)\in\mathbb{R}^{n\times n}$ is
\begin{align}
P^{(n)}(a,b) = \det\underbrace{\left(\begin{matrix}
a & -1 & 0 & \dots & 0 & 0\\
-1 &  2 &-1& \dots  & 0 & 0\\
0 & -1 & 2 & \dots  & 0 & 0\\
\vdots & &\vdots & \vdots & \vdots & \vdots\\
0 & 0 & 0 & \dots & 2 & -1\\
0 & 0 & 0 & \dots &-1  & b
\end{matrix}\right)}_{n\times n}.
\end{align}

(1) We first prove that if $a>\frac{n-2}{n-1}$, then the first $n-1$ principal minors of $\mathcal{D}^{(n)}\in\mathbb{R}^{n\times n}$, namely $Q^{(n)}_k(a)$ ($1\leq k\leq n-1)$ are all positive. 

For $n=3$, with direct computations, one can check that if $ a>\frac{3-2}{3-1}=\frac{1}{2}$ then $Q^{(3)}_1(a)$ and $Q^{(3)}_2(a)$ are positive.
Now, for the induction case $n-1$, we assume if $a>\frac{n-3}{n-2}$, then the first $n-2$ principal minors of  $\mathcal{D}^{(n-1)}$, namely $Q^{(n-1)}_k(a)$ $(1\leq k\leq n-2)$, are  all positive. With this induction assumption, we will prove that if $a>\frac{n-2}{n-1}$ then the first $n-1$ principal minors of $\mathcal{D}^{(n)}(a,b)$ are all positive. 

 The definition in \eqref{eq:Qan_def} implies that $Q^{(n)}_k(a)=Q^{(n-1)}_k(a)$ when $1\leq k\leq n-2$. Moreover, if $a>\frac{n-2}{n-1}$, then $a>\frac{n-3}{n-2}$ and the induction assumption leads to 
$$Q^{(n)}_k(a)=Q^{(n-1)}_k(a)>0,\; 1\leq k\leq n-2.$$ 
Now, we only need to prove that $Q^{(n)}_{n-1}$ is positive:
\begin{align}
Q^{(n)}_{n-1}(a)&=\det\left(\begin{matrix}
a & -1 & 0 & \dots & 0 & 0\\
-1 &  2 &-1& \dots  & 0 & 0\\
0 & -1 & 2 & \dots  & 0 & 0\\
\vdots & &\vdots & \vdots & \vdots & \vdots\\
0 & 0 & 0 & \dots & 2 & -1\\
0 & 0 & 0 & \dots &-1  & 2
\end{matrix}\right)
=\det\left(\begin{matrix}
a & -1 & 0 & \dots & 0 & 0\\
0 &  2-\frac{1}{a} &-1& \dots  & 0 & 0\\
0 & -1 & 2 & \dots  & 0 & 0\\
\vdots & &\vdots & \vdots & \vdots & \vdots\\
0 & 0 & 0 & \dots & 2 & -1\\
0 & 0 & 0 & \dots &-1  & 2
\end{matrix}\right)\notag\\
&= aQ^{(n-1)}_{n-2}\left(2-\frac{1}{a}\right).
\end{align}

Moreover, if $a>\frac{n-2}{n-1}>0$, then
\begin{align}
2-\frac{1}{a}>2-\frac{n-1}{n-2}=\frac{n-3}{n-2},
\end{align}
and together with the induction assumption for $n-1$, we have
$$Q^{(n)}_{n-1}(a)= aQ^{(n-1)}_{n-2}\left(2-\frac{1}{a}\right)>0.$$

(2) At last, we prove if $ a > \frac{(n-2)b-(n-3)}{(n-1)b-(n-2)}$ and $b>\frac{n-2}{n-1}$, the last principal minor $P^{(n)}(a,b)$ is positive. We perform an induction proof with respect to $n$.

For the base case, when $n=3$, $b>\frac{n-2}{n-1}=\frac{1}{2}$ and $a > \frac{(n-2)b-(n-3)}{(n-1)b-(n-2)}=\frac{b}{2b-1}$. A direct computation shows that $P^{(3)}(a,b)$ is positive.
Now, we turn to the induction case, for $n-1$, assume that 
\begin{align}
b>\frac{n-2}{n-3}\quad\text{and}\quad a > \frac{(n-3)b-(n-4)}{(n-2)b-(n-3)}\label{eq:induction_assumption}
\end{align} implies
that $P^{(n-1)}(a,b)$ is positive. 

We note that $P^{(n)}(a,b)$ can be related to $P^{(n-1)} \left(2-\frac{1}{a},b\right)$ through the rules of determinants as follows:
\begin{align}
P^{(n)}(a,b)&=\det\left(\begin{matrix}
a & -1 & 0 & \dots & 0 & 0\\
-1 &  2 &-1& \dots  & 0 & 0\\
0 & -1 & 2 & \dots  & 0 & 0\\
\vdots & &\vdots & \vdots & \vdots & \vdots\\
0 & 0 & 0 & \dots & 2 & -1\\
0 & 0 & 0 & \dots &-1  & b
\end{matrix}\right)
=\det\left(\begin{matrix}
a & -1 & 0 & \dots & 0 & 0\\
0 &  2-\frac{1}{a} &-1& \dots  & 0 & 0\\
0 & -1 & 2 & \dots  & 0 & 0\\
\vdots & &\vdots & \vdots & \vdots & \vdots\\
0 & 0 & 0 & \dots & 2 & -1\\
0 & 0 & 0 & \dots &-1 & b
\end{matrix}\right)\notag\\
&= a P^{(n-1)} \left(2-\frac{1}{a},b\right).
\end{align}
Now, we just need to verify that $a>0$ and $P^{(n-1)} \left(2-\frac{1}{a},b\right)>0$ under the assumption that $b>\frac{n-2}{n-1}$ and $a > \frac{(n-2)b-(n-3)}{(n-1)b-(n-2)}$. As $b>\frac{n-2}{n-1}$ and $a> \frac{(n-2)b-(n-3)}{(n-1)b-(n-2)}$, we have 
\begin{align*}
&(n-1)b-(n-2)>0,\quad(n-2)b-(n-3)\geq\frac{(n-2)^2-(n-1)(n-3)}{n-1}\frac{1}{n-1}>0,\\
&0<\frac{1}{a}<\frac{(n-1)b-(n-2)}{(n-2)b-(n-3)}\quad
\text{and}\quad a^*=2-\frac{1}{a}>2-\frac{(n-1)b-(n-2)}{(n-2)b-(n-3)}=\frac{(n-3)b-(n-4)}{(n-2)b-(n-3)}.
\end{align*}
Hence, we have $b=\frac{n-2}{n-1}>\frac{n-2}{n-3}$, and $a^*>\frac{(n-3)b-(n-4)}{(n-2)b-(n-3)}$. The induction assumption for the $n-1$ case in \eqref{eq:induction_assumption} is satisfied and $P^{(n-1)} \left(2-\frac{1}{a},b\right)=P^{(n-1)} \left(a^*,b\right)>0$.

\bibliographystyle{plain}
\bibliography{appelo,pzc,shuang}
\end{document}